\documentclass[reqno]{amsart} \usepackage{bbm}
\usepackage[top=2.8cm,left=3.5cm,right=3.5cm,bottom=2.8cm]{geometry}
\usepackage{amsfonts}
\usepackage{amsmath,amsthm,amscd,mathrsfs,amssymb,graphicx,enumerate,
	dsfont}
\usepackage{xypic}
\usepackage{hyperref}
\usepackage[usenames,dvipsnames]{xcolor}
\hypersetup{colorlinks=true,citecolor=NavyBlue,linkcolor=Brown,urlcolor=Orange}

\newtheorem{thm2}{Theorem}
\newtheorem{cor2}{Corollary}
\newtheorem{conj2}{Conjecture}
\newtheorem{defn2}{Definition}
\newtheorem{thm}{Theorem}[section]
\newtheorem{cor}[thm]{Corollary}
\newtheorem{lem}[thm]{Lemma}
\newtheorem{prop}[thm]{Proposition}
\newtheorem{conj}[thm]{Conjecture}
\theoremstyle{definition}
\newtheorem{defn}[thm]{Definition}

\newtheorem{ex}[thm]{Example}
\newtheorem{rmk}[thm]{Remark}

\newtheorem{ques}[thm]{Question}

\DeclareMathOperator{\End}{End} 
\DeclareMathOperator{\Hom}{Hom}

\newcommand{\C}{\ensuremath\mathds{C}}
\newcommand{\R}{\ensuremath\mathrm{R}}
\newcommand{\Z}{\ensuremath\mathds{Z}}
\newcommand{\Q}{\ensuremath\mathds{Q}}

\newcommand{\HH}{\ensuremath\mathrm{H}}
\newcommand{\CH}{\ensuremath\mathrm{CH}}
\newcommand{\DCH}{\ensuremath\mathrm{DCH}}

\begin{document}

	\title[On Bloch's conjecture for abelian varieties]{Generic cycles, Lefschetz
		representations, and the generalized Hodge and Bloch conjectures\linebreak
		for abelian varieties}
	\author{Charles Vial}
	
	
	
	\address{
		Universit\"at Bielefeld, Germany}
	\email{vial@math.uni-bielefeld.de}
	
	\date{}
	
	\begin{abstract} 
		We prove Bloch's conjecture for correspondences on powers of complex
		abelian varieties, that are ``generically defined''. As an application we
		establish vanishing results for (skew-)symmetric cycles on powers of abelian
		varieties and we
		address a question of Voisin concerning (skew-)symmetric cycles on powers
		of K3 surfaces in the case of Kummer surfaces. We also prove  Bloch's
		conjecture in the following situation. Let $\gamma$ be a correspondence
		between two  abelian varieties $A$ and $B$ that can be written as
		a linear combination of products of symmetric divisors. Assume that $A$
		is isogenous to the product of an abelian variety of totally real type with
		the power of an
		abelian surface. We show that $\gamma$ satisfies the conclusion of Bloch's
		conjecture. A key ingredient consists in establishing a strong form of the
		generalized Hodge conjecture for Hodge sub-structures of the cohomology of $A$
		that arise as sub-representations of the Lefschetz group of $A$. As a
		by-product of our method, we use a strong form of the generalized Hodge
		conjecture established for powers of abelian surfaces to show that every
		finite-order symplectic automorphism of a
		generalized Kummer variety acts as the identity on the zero-cycles.
	\end{abstract}
	
	\maketitle

	\vspace{-10pt}
	\section*{Introduction}
	
	Throughout this note, Chow groups are with rational coefficients.
	Let $X$ be a smooth projective complex variety of dimension $d$ and let
	$\gamma$
	be a correspondence in $\CH^d(X\times X)$ such that $\gamma_*\CH_0(X)=0$. The
	Bloch--Srinivas argument \cite{bs} implies  that $\gamma^*\HH^{*}(X,\Q)$ is
	supported on a divisor, which in turn implies that  $\gamma^*\HH^{i,0}(X) = 0$
	for
	all integers $i$.  The Bloch conjecture stipulates that, conversely,  should $\gamma^*\HH^{i,0}(X)$ vanish
	for
	all integers $i$, then $\gamma$ acts nilpotently on $\CH_0(X)$. (In fact, the conjecture predicts that $\gamma$ should act as zero on the graded pieces of the conjectural Bloch--Beilinson filtration on $\CH_0(X)$.)
	
	More generally, if 	$\gamma_*\CH_r(X) = 0$ for all $r<n$, then the
	Bloch--Srinivas
	argument  implies  that $\gamma^*\HH^{*}(X,\Q)$ is supported on a subscheme of
	codimension~$n$, which in turn implies that $\gamma^*\HH^{i,j}(X) = 0$
	for	all integers $i$ and $j<n$. 
	The generalized
	Bloch conjecture is the following converse assertion\,:
	
	\begin{conj2}[Generalized Bloch conjecture]\label{conj:genBloch}
		Let $X$ be a smooth projective complex variety of dimension $d$, and let $\gamma \in
		\CH^d(X \times X)$ be a correspondence. Suppose that  $\gamma^*\HH^{i,j}(X)=0$
		for
		all $j<n$, or, equivalently in terms of the Hodge coniveau
		filtration, $\gamma^*
		\HH^*(X,\Q) \subseteq \mathrm{N}_H^n\HH^*(X,\Q)$. Then $\gamma_*$ acts nilpotently on $\CH_r(X)$
		for
		all $r<n$.
	\end{conj2}
	
	The conjecture is wide open, but has notably been established for surfaces with
	$\HH^{2,0}=0$ not of general type \cite{bkl}, for certain 
	surfaces with $\HH^{2,0}=0$ of general type \cite{voisingodeaux,
		voisincatanese}, and for finite-order symplectic automorphisms of K3 surfaces
	\cite{voisink3, huybrechts}. 
	\medskip
	
	Conjecture \ref{conj:genBloch} follows from the combination of (a) the generalized Hodge conjecture (for smooth projective varieties, and not just for $X$) and~(b) the existence of the conjectural Bloch--Beilinson filtration. Indeed, if $\gamma^*\HH^{i,j}(X)=0$
	for	all $j<n$, then the generalized Hodge conjecture for~$X$ implies that $\gamma^*\HH^{*}(X,\Q)$ is supported on a closed subscheme $X$ of codimension~$n$. By~\cite{andreIHES}, the standard conjectures (for $\tilde Z\times X$, where $\tilde Z \to Z$ is a desingularization) then provide a self-correspondence $p \in \CH^d(X\times X)$ supported on $Z\times X$ such that $p$ induces in cohomology a projector with image $\gamma^*\HH^{*}(X,\Q)$ (see Conjecture~\ref{conj:strongGHC}). It follows that $\gamma \circ (\Delta_X - p)$ acts trivially on $\HH^*(X,\Q)$, \emph{i.e.}, that $\gamma \circ (\Delta_X - p)$ is homologically trivial.
	In particular, if $\mathrm{F}^\bullet$ denotes the conjectural Bloch--Beilinson filtration, $\gamma_* \circ (\Delta_X - p)_*$ sends $\mathrm{F}^l\CH_r(X)$ to $\mathrm{F}^{l+1}\CH_r(X)$ for all $l$ and all $r$. Since conjecturally $\mathrm{F}^{r+1}\CH_r(X) = 0$ for all $r$, we find that $\gamma_* \circ (\Delta_X - p)_*$ is nilpotent. (Alternately, the conjectural Kimura--O'Sullivan finiteness for $X$ implies that $\gamma \circ (\Delta_X - p)$ is nilpotent). 
	 Finally, for support reasons, $p_*$ acts as zero on $\CH_r(X)$ for $r<n$, and we conclude that $\gamma_* $ acts nilpotently on $\CH_r(X)$ for $r<n$. 
	 
	  We note that by applying Conjecture \ref{conj:genBloch} to $\gamma = \Delta_X - \sum_i \pi_\mathrm{alg}^{2i}$, where the $\pi_\mathrm{alg}^{2i}$ are projectors on the degree-$2i$ Hodge classes (which conjecturally exist), one recovers the more classical version of the generalized Bloch conjecture stated \emph{e.g.} in \cite[Conj.~1.9]{voisinbook}. We also note that one cannot conclude in general that $\gamma_* $ acts as zero on $\CH_r(X)$ for $r<n$. Consider indeed a smooth projective curve $C$ of positive genus and the correspondence $\gamma = C\times \alpha$, where $\alpha$ is a non-zero degree-$0$ $0$-cycle on $C$\,; then $\gamma^*\HH^*(C,\Q) = 0$ and $(\gamma \circ \gamma)_*\CH_0(C)= 0$, but $\gamma_*\CH_0(C) = \Q \alpha \neq 0$. \medskip
	
	\vspace{-10pt}
	Our	main results are Theorem~\ref{T:main2} and Theorem~\ref{T:mainLef}. We
	establish the generalized Bloch
	conjecture for certain correspondences between abelian
	varieties, that are of two types\,: either
	``generically defined'', or belong to the sub-algebra generated by symmetric
	divisors (with some further assumptions on the abelian varieties). In both cases, the strategy consists in first showing  that the
	Hodge sub-structure $\gamma^*
	\HH^*(X,\Q)$ is supported in codimension~$n$ in a strong sense (existence of a cycle $p\in \CH^d(X\times X)$ as in the discussion above with additional properties), in particular
	that the generalized Hodge conjecture for $\gamma^*\HH^*(X,\Q)$ holds\,; see
	Propositions~\ref{P:niveau} and~\ref{prop:projlef}.
	For that matter, we formulate in Conjecture \ref{conj:strongGHC} a strong (but equivalent, when considered for all complex smooth projective varieties) version of the generalized Hodge conjecture.
	 This information on the cohomological support of $\gamma$ is then lifted to rational equivalence thanks either to Kimura-O'Sullivan finite-dimensionality (Theorem~\ref{T:Kimura}) or to a recent result of O'Sullivan (Theorem~\ref{T:Osullivan}). In the latter case, that is, when $\gamma$ is in addition \emph{symmetrically distinguished} (see \S \ref{sec:symdist}), then one can conclude that $\gamma_*\CH_r(X) = 0$
	for
	all $r<n$ (see Theorems~\ref{T:main2}(2) and~\ref{T:mainLef}).\medskip

	\vspace{-5pt}
	\subsection{Generically defined cycles}
	A \emph{generically defined cycle} on the $m$-fold power of a
	polarized complex abelian variety $A$ of degree $d^2$ and dimension $g$ is a
	cycle (with rational coefficients) in $\CH^*(A^m)$ that is the
	restriction, for some integer $N\geq 3$ (in fact, by Remark \ref{rmk:N},
	\emph{for any} integer $N\geq 3$), of a cycle on the $m$-fold power of the
	universal polarized abelian
	variety of degree $d^2$ and dimension $g$ with level-$N$ structure\,; see
	Definition~\ref{D:unidef}.
	A generically
	defined self-correspondence on the $m$-fold power of complex polarized abelian
	varieties
	of degree $d^2$ and dimension $g$ is a generically defined cycle of codimension
	$mg$ on the
	$2m$-fold power of polarized complex abelian varieties  of degree $d^2$ and  of
	dimension~$g$.\medskip
	
	Our first main result is Theorem~\ref{T:main2}, a special instance of which is the following\,:
	
	\begin{thm2}\label{T:main}
		Suppose that $\gamma$ is a generically defined correspondence on
		the
		$m$-fold power of polarized complex abelian varieties. Assume
		that $\gamma^*\HH^{i,j}(A^m)=0$ for all $j<n$
		for some (equivalently, for all) polarized complex abelian variety $A$ of
		dimension $g$ and degree~$d^2$. Then
		$\gamma_*$ acts nilpotently on $\CH_r(A^m)$ for all $r<n$.
	\end{thm2}

	The proof  consists in first establishing Theorem
	\ref{T:main} for a very general complex abelian variety~$A$. For such a
	variety, a
	strong form of the generalized Hodge conjecture (as in Conjecture \ref{conj:strongGHC}) holds (Hazama's
	Theorem~\ref{T:gen2})
	and
	makes it possible to interpret  the $n$-th Hodge coniveau part (see Definition
	\ref{D:Hconiveau}) $\mathrm{N}_H^n \HH^*(A^m,\Q)$  as a ``generically defined''
	sub-motive of $A^m$ whose Tate twist by $n$ is effective and, in fact,
	isomorphic to a direct summand of a finite direct sum of motives of $A^m$\,;
	see the key
	Proposition
	\ref{P:niveau}. One can conclude by using the Kimura finite-dimensionality
	\cite{kimura} of motives of abelian varieties. One establishes Theorem
	\ref{T:main} for all abelian
	varieties by specialization.
	\medskip
	
	Recall from Beauville \cite{beauville} that the Chow group of zero-cycles on an abelian variety $A$ of dimension $g$ splits into eigenspaces as $$\CH_0(A) = \CH_0(A)_{(0)} \oplus \cdots \oplus \CH_0(A)_{(g)},$$ where $\CH_0(A)_{(i)} = \{a \in \CH_0(A) : [n]_*a = n^ia \ \text{for all} \ n\in \Z \}$ with $[n]:A\to A$ the multiplication-by-$n$ homomorphism.
	As an application of Theorem \ref{T:main}, we obtain\,:
	
	\begin{cor2}[Theorem \ref{T:sym}]
	Let $A$ be an abelian variety of dimension $g$, and let $i$ be a nonnegative
integer.
Let $N > \binom{g}{i}$ and suppose that $a_j$, $1\leq j \leq N$, are
zero-cycles on $A$ such that $[n]_*a_j = n^i a_j$ for all integers $n$. Then the following holds.
\begin{itemize}
	\item For $i$ odd, the symmetrization of $ a_{1}
	\times \cdots\times a_{N} $ vanishes, \emph{i.e.} $$\sum_{\sigma \in \mathfrak{S}_N} a_{\sigma(1)}
	\times \cdots\times a_{\sigma(N)} = 0  \quad \text{in}\ \CH_0(A^N).$$
	\item For $i$ even, the anti-symmetrization of $ a_{1}
	\times \cdots\times a_{N} $ vanishes, \emph{i.e.} $$\sum_{\sigma \in \mathfrak{S}_N}
	\mathrm{sgn}(\sigma)\,  a_{\sigma(1)} \times \cdots\times a_{\sigma(N)} = 0  \quad \text{in}\ 
	\CH_0(A^N).$$
\end{itemize}
	\end{cor2}

	\subsection{Lefschetz sub-representations}\label{sec:lefintro}

	Let $A$ be an abelian variety. We define $$\mathrm{R}^*(A)\subseteq \CH^*(A)$$ to be the
	sub-algebra
	of $\CH^*(A)$ generated by symmetric divisors. Note that if  $B$ is another
	abelian variety, then the class of the graph of any homomorphism $A\to B$
	belongs to $\R^*(A\times B)$ (see Proposition~\ref{prop:R}).
	As a link to Theorem~\ref{T:main}, we note that all generically defined cycles on the
	$m$-fold power of an abelian
	variety
	$A$ that we consider in explicit examples belong to  $\R^*(A^m)$\,; 
  see however Question~\ref{R:gdsd}. We can prove the generalized Bloch
	conjecture for correspondences that belong to $\R^*$ on certain abelian
	varieties
	(which are not necessarily very general).
	
	\begin{defn2}[Abelian varieties of totally real type] \label{def:totreal} 
		An abelian variety $A$ is said to be of \emph{totally real type} if the center
		of its endomorphism ring $\End^0(A):=\End(A)\otimes_{\Z}\Q$ is isomorphic to a product of totally
		real
		fields. Equivalently, $A$ is of totally real type if it is isogenous to
		$A_1^{m_1} \times \cdots \times A_s^{m_s}$ with the $A_i$ simple of type I,
		II,
		or III (see \S\ref{sec:albertlef}).
	\end{defn2}	
	
	Our second main result is Theorem \ref{T:mainLef}, a special instance of which is the following\,:	
	
	\begin{thm2}\label{T:main3}
		Let $A$ and $B$ be two abelian varieties, and let $\gamma$ be a cycle in
		$\mathrm{R}^*(A\times B)$. Suppose that $A$ is  of totally real type.
		If
		$\gamma^*\HH^{i,j}(B)=0$ for all $j<n$, then
		$\gamma_*\CH_r(A) = 0$ for all $r<n$.
	\end{thm2}	
	
	There are two main arguments entering the proof of Theorem \ref{T:main3}. First, as explained in \S
	\ref{sec:lefschetz},
	the fact that $\gamma$ belongs to $\mathrm{R}^*(A\times B)$ implies that the
	Hodge sub-structure $\gamma^*\HH^*(B,\Q)$ is a sub-representation of the
	\emph{Lefschetz group} of $A$ acting on $\HH^*(A,\Q)$. The first step does not consist in establishing the
	generalized Hodge conjecture for $A$ but, instead, consists in
	showing that any sub-representation of the Lefschetz group of $A$ acting on
	$\HH^*(A,\Q)$ satisfies a strong form of the generalized Hodge conjecture (as in Conjecture~\ref{conj:strongGHC})\,;
	see Proposition~\ref{prop:projlef}. We note that if $A$ is a very general complex abelian variety, then $\End^0(A)=\Q$ and, by coincidence of the Lefschetz group of $A$ with its Hodge group, every Hodge sub-structure of $\HH^*(A^m,\Q)$ is a  Lefschetz sub-representation. The generalized Hodge conjecture for self-powers of the very general complex abelian variety was established by Hazama \cite{hazama} (see Theorem \ref{T:gen2}). By shifting our attention to Lefschetz sub-representations, we can generalize the aforementioned result of Hazama (we refer to
	Theorem \ref{thm:lefhodge} for a more precise statement)\,:

	\begin{thm2}[strong GHC for Lefschetz sub-representations of abelian varieties of
		totally real type]
		Let $A$ be a complex abelian variety, and let $H \subseteq \HH^k(A,\Q)$ be a
		Lefschetz sub-representation of Hodge level $\leq k-2n$. Suppose that $A$ is
		of totally real type.
		Then $H$ satisfies the strong generalized Hodge conjecture \ref{conj:strongGHC}, in particular, $H$ is supported on a closed subset of codimension~$n$.
	\end{thm2}

	In the second step, instead of using Kimura's finite-dimensionality which would only yield that $\gamma$ acts nilpotently on $\CH_r(A)$ for all $r<n$, we utilize a recent powerful result
	of
	O'Sullivan~\cite{o'sullivan} which in particular implies that the ring $\R^*(A)$ injects into
	cohomology for all abelian varieties~$A$. 
	We refer to the proof of Theorem \ref{T:mainLef} for the details.
	
	Unfortunately, our method for establishing  (a strong form of) the generalized Hodge conjecture
	for Lefschetz sub-representations of $\HH^i(A^m,\Q)$ for $A$ of totally real
	type does not seem to extend in a direct way to the  interesting case of abelian varieties of type IV or even to that of abelian
	varieties of CM type\,; see Remark \ref{rmk:limit}.
	As far as we know,
	the conjecture is still open for the product of four pairwise non-isogenous CM
	elliptic curves. 
	
	Nonetheless, the generalized Hodge conjecture was
	established by Abdulali \cite{abdulali} for powers of a simple abelian surface
	of CM type 
	(see Theorem
	\ref{thm:GHC}). Abdulali's proof yields a strong form of the  generalized Hodge conjecture (as in Conjecture~\ref{conj:strongGHC})
	for powers of abelian varieties of dimension~$\leq 2$ (see Corollary
	\ref{cor:GHCsurface}). Using Abdulali's theorem, we establish in Theorem~\ref{T:mainLef} a slightly
	more general version of Theorem~\ref{T:main3} by allowing $A$ to be isogenous to the product of an abelian
	variety of totally real type with either the power of a CM abelian surface or a
	product of powers of three CM elliptic curves. Again the key input consists in establishing a strong form of the generalized Hodge conjecture for Lefschetz sub-representations (Proposition~\ref{prop:projlef}).
	Since the case of powers of abelian surfaces is particularly telling due to the
	link with so-called generalized Kummer varieties, we single out the following
	statement from Theorem~\ref{T:mainLef}\,:
	
	\begin{thm2}\label{T:mainsurface}
		Let $A$ and $B$ be two abelian varieties, and let $\gamma$ be a cycle in
		$\mathrm{R}^*(A\times B)$. Suppose that $A$ is isogenous to a power of an
		abelian variety of dimension $\leq 2$.
		If
		$\gamma^*\HH^{i,j}(B)=0$ for all $j<n$, then
		$\gamma_*\CH_r(A) = 0$ for all $r<n$.
	\end{thm2}

	\subsection{Applications} Section~\ref{S:ex} is concerned with concrete applications of the above results. Specifically, 
	Theorems
	\ref{T:sym}, \ref{T:Lin} and \ref{T:Alt2Kum} provide vanishing results for (skew)-symmetric cycles on
	powers of abelian varieties and generalized Kummer varieties, while Theorem~\ref{T:voisin} settles a conjecture of
	Voisin about K3 surfaces in the case of Kummer surfaces. All these results are proved as consequences of Theorem \ref{T:main}, so that the reader interested only in those can skip reading Section \ref{sec:LR} entirely.
	Finally, in \S \ref{sec:motab} we establish a variant of Theorem~\ref{T:mainsurface} (Theorem~\ref{T:app}) which we use in \S \ref{sec:symplecto} to 
	show that
	a finite-order symplectic automorphism of a generalized Kummer variety acts as
	the
	identity on the Chow group of zero-cycles (see also Proposition~\ref{prop:genkuminduced})\,:
	\begin{thm2}[Theorem~\ref{thm:sympGK}]
		Let $A$ be an abelian surface. If $f$ is a finite-order symplectic
		automorphism of
		the
		generalized Kummer variety $K_n(A)$, then $f_* : \CH_0(K_n(A)) \to
		\CH_0(K_n(A))$ is the identity map.
	\end{thm2}

\vspace{4pt}

\noindent \textbf{Acknowledgments.} I would like to thank  Robert
Laterveer for bringing to my attention the questions posed by Voisin in
\cite[\S 3]{voisin0}, and Giuseppe Ancona for very useful discussions.

	\section{Preliminaries}
	
	\subsection{Polarized abelian varieties, and level structures}\label{S:def}
	A polarization $L$ on a complex abelian variety $A$ of dimension $g$ is by
	definition the first Chern class of an ample line bundle $\mathcal{L}$. We
	denote $\hat A = \mathrm{Pic}^0(A)$ the dual abelian variety and
	$\mathcal{P}_A$
	the Poincar\'e line-bundle on $A\times \hat A$. 
	Let $\phi_{\mathcal{L}} : A \to \hat{A}$ be the morphism given on points by $a
	\mapsto t_a^*\mathcal{L} \otimes \mathcal{L}^{-1}$. 
	By definition, the degree of the polarization $L$ is the degree of the isogeny
	$\phi_{\mathcal{L}} : A \to \hat{A}$\,; it is a square since we have $\deg(L) =
	\chi(\mathcal{L})^2$. 
	We will often view the Poincar\'e line-bundle $\mathcal{P}_A$ as a line-bundle
	on
	$A\times A$, by pulling back along $\mathrm{id}_A\times \phi_{\mathcal{L}}$ the
	Poincar\'e line-bundle on $A\times \hat{A}$.
	The Fourier--Mukai transform of $\mathcal{L}$ is the sheaf
	$\mathcal{F}(\mathcal{L}) := p_{2,*}(\mathcal{P}_A\otimes p_1^*\mathcal{L})$\,;
	it is a vector-bundle on $\hat A$. The dual polarization $\hat L$ on $\hat A$
	is
	the first Chern class of $\mathrm{det}( \mathcal{F}(\mathcal{L}) )^{-1}$\,; see
	\cite{bl}.
	
	Denote $\iota_\Delta : A \to A\times A$ the diagonal embedding.
	We define a correspondence $\Lambda_A^i$ in $\CH_i(A\times A)$ as follows\,:
	\begin{equation}\label{E:Lambda}
		\Lambda_A^i = \left\{\begin{array}{ll} \iota_{\Delta,*} L^{g-i}  & \mbox{if
				$i\leq g$\,;}\\
			\hat{\mathcal{F}} \circ (\iota_{\Delta,*} {\hat L}^{i-g}) \circ \mathcal{F} &
			\mbox{if $i >g$.}
		\end{array}\right.
	\end{equation}
	Note that $\Lambda_A^i$  induces an isomorphism $\HH^i(A,\Q)
	\stackrel{\simeq}{\longrightarrow} \HH^{2g-i}(A,\Q)$.
	
	Following \cite[\S 6]{mumford}, for a projective abelian scheme $\mathcal{A}
	\to S$ over a 
	Noetherian scheme $S$, we define its dual $\hat{\mathcal{A}} \to S$ to be the
	projective scheme that is 
	the open sub-group-scheme of $\mathrm{Pic}(\mathcal{A}/S)$ whose geometric
	points correspond to the invertible sheaves some power of which are
	algebraically equivalent to zero, and we define a \emph{polarization} on
	$\hat{\mathcal{A}} \to S$ to be a $S$-homomorphism $\mathcal{A} \to \hat{
		\mathcal{A}}$ such that, for all geometric points $\bar{s}$ of $S$, the
	induced
	$\mathcal{A}_{\bar{s}} \to \hat{\mathcal{A}}_{\bar{s}}$ is of the form
	$\phi_{\mathcal{L}}$ for some ample line-bundle $\mathcal{L}$ on
	$\mathcal{A}_{\bar{s}}$.
	
	Let $\mathcal{A} \to S$ be a projective abelian scheme of relative dimension
	$g$ over a 
	Noetherian scheme $S$, and let $N$ be an integer $\geq 2$. Assume that the
	characteristics of the residue fields of all closed points of $S$ do not divide
	$N$.  	A \emph{level-$N$ structure} on $\mathcal{A} \to S$ consists of $2g$
	sections $\sigma_1,\ldots, \sigma_{2g}$ of $\mathcal{A} \to S$ such that their
	restriction to any geometric point $\bar s$ of $S$ provide a basis of the
	$N$-torsion of the fiber of $\mathcal{A} \to S$ over $\bar s$, and such that
	$[N]\circ \sigma_i = 0_{\mathcal A}$ for all $i$, where $[N]$ denotes the
	multiplication-by-$N$ morphism and where $0_{\mathcal A}$ is the identity
	section of $\mathcal{A} \to S$.

	\subsection{Motives of abelian varieties, symmetrically distinguished cycles} \label{sec:symdist}
	We will use freely the language of Chow motives, as is described for instance
	in \cite{andre}. The unit motive is denoted~$\mathds{1}$ and the motive of a
	smooth projective variety is denoted $\mathfrak{h}(X)$. Our convention for the
	Tate twist is such that $\mathfrak{h}(\mathds{P}^1) = \mathds{1} \oplus
	\mathds{1}(-1)$.
	
	The Chow motives of abelian varieties have particularly nice properties. First
	they are finite-dimensional in the sense of Kimura \cite{kimura}. Without going
	into the details of Kimura's notion of finite-dimensionality, let us only
	mention the following property\,:
	
	\begin{thm}[Kimura \cite{kimura}] \label{T:Kimura}
		Let $A$ be a complex abelian variety of dimension $g$, and let $\Gamma \in
		\CH^g(A\times A)$ be a self-correspondence on $A$. Assume that $\Gamma$ is
		numerically trivial. Then
		$\Gamma$
		is nilpotent, \emph{i.e.}, there exists a positive integer $N$ such that
		$\Gamma^{\circ N} = 0$ in $ \CH^g(A\times A)$.
	\end{thm}

	Second, O'Sullivan \cite{o'sullivan} has recently identified a sub-algebra of
	$\CH^*(A)$ consisting of cycles that are called \emph{symmetrically
		distinguished} (see \cite[p.2]{o'sullivan} for a definition),  with the
	following property\,:
	
	\begin{thm}[O'Sullivan \cite{o'sullivan}] \label{T:Osullivan} Let $A$ be a
		complex abelian variety.
		The  symmetrically distinguished cycles in
		$\CH^*(A)$ form a graded $\Q$-sub-algebra, denoted $\DCH^*(A)$, that contains
		symmetric divisors and that is
		stable under pull-backs and push-forwards along homomorphisms of abelian
		varieties. Moreover
		the composition $$\DCH^*(A)\hookrightarrow \CH^*(A)\twoheadrightarrow
		\overline
		\CH^*(A)$$ is an isomorphism of $\Q$-algebras. Here, $ \overline
		\CH^*(A)$ denotes the Chow ring of $A$ modulo numerical equivalence. In
		particular, a symmetrically distinguished cycle that is homologically trivial
	 is rationally trivial.
	\end{thm}

	The following definition will be relevant to our work concerned with Lefschetz
	representations\,; see e.g. Lemma~\ref{lem:projlef}.
	
	\begin{defn}\label{def:R}
		For a complex abelian variety $A$, we denote $$\R^*(A) \subset \CH^*(A)$$ the
		$\Q$-sub-algebra generated by symmetric divisors and we denote
		$\overline{\R}^*(A)$ its image in $\overline \CH^*(A)$, or equivalently, since homological and numerical equivalence agree on complex abelian varieties, its image in $\HH^*(A,\Q)$ under the cycle class map. 
	\end{defn}
	By
	O'Sullivan's Theorem~\ref{T:Osullivan}, $\R^*(A)$  is a
	sub-algebra of $\DCH^*(A)$ that maps isomorphically onto $\overline{\R}^*(A)$
	via the cycle class map\footnote{That $\R^*(A)$   maps isomorphically onto
		$\overline{\R}^*(A)$ was also established independently by
		Ancona~\cite{ancona2}
		and Moonen~\cite{moonen}.}. Note that a polarization of $A$ is a symmetric
	divisor
	on $A$, and that the first Chern class of the Poincar\'e line-bundle is a
	symmetric divisor on $A\times \hat{A}$. 
	We note that by Proposition \ref{prop:R} below the cycles $\Lambda_A^i$
	of \eqref{E:Lambda} belong to $\R^*(A\times A)$.
	
	\subsection{Hodge structures and the generalized Hodge conjecture} A $\Q$-Hodge
	structure $H$ is a rational vector space of finite dimension together with a
	decomposition of $H_\C := H\otimes_\Q \C$ as a direct sum of complex linear
	subspaces $H^{p,q}$ for integers $p,q$ such that $\overline{H^{p,q}} = H^{q,p}$
	and such that the grading by $p+q$, called the weight grading, is defined over
	$\Q$. 
	The \emph{level} of a Hodge structure $H$  is defined as  $$\ell(H) :=
	\max\{|p-q| :
	H^{p,q}\neq 0 \},$$ with the convention that we declare $H=0$ to have level
	$-\infty$. A Hodge structure $H$ is said to be \emph{effective} if
	$H^{p,q} = 0$ for $p<0$.

	\begin{defn}\label{D:Hconiveau}
		Let $H$ be a rational Hodge structure of weight $k$. The \emph{Hodge coniveau
			filtration} is 
		$$\mathrm{N}_H^n H = \text{the largest Hodge sub-structure of $H$ of level
			$\leq k-2n$}.$$
		In other words, $\mathrm{N}_H^n H$  is the largest Hodge sub-structure $H'$ of
		$H$ such that  $H'\otimes \Q(n)$ is effective. Here $\Q(n)$ denotes the
		$1$-dimensional Hodge structure of weight $-2n$ and level~$0$.
	\end{defn}
	
	\begin{conj}[Grothendieck's generalized Hodge conjecture]\label{conj:GHC}
		Let $X$ be a complex smooth projective variety. If $H$ is a sub-Hodge
		structure
		of $\HH^k(X,\Q)$ of level $\leq k-2n$, \emph{i.e.} $H\subseteq
		\mathrm{N}_H^{n}\HH^k(X,\Q)$, then $H$ is supported in codimension $n$, \emph{i.e.}
		there exists a closed subscheme $Z\subseteq X$ of codimension $n$ such that
		$H$
		is mapped to zero under the restriction homomorphism $\HH^k(X,\Q) \to
		\HH^k(X\backslash Z,\Q)$.
	\end{conj}

	Combining the above with the standard conjectures, a theorem of Yves Andr\'e \cite{andreIHES} on motivated cycles allows us to formulate the following conjecture (see \emph{e.g.} the proof of \cite[Prop.~4.1]{ACMV}).
	
	\begin{conj}[Strong form of the generalized Hodge conjecture]\label{conj:strongGHC}
	Let $X$ be a complex smooth projective variety of dimension $d$. If $H$ is a sub-Hodge
structure
of $\HH^k(X,\Q)$ of level $\leq k-2n$, \emph{i.e.} $H\subseteq
\mathrm{N}_H^{n}\HH^k(X,\Q)$, then 
there exists a closed subscheme $Z\subseteq X$ of codimension $n$ and a correspondence  $p\in \CH^d(X\times X)$ supported on $Z\times X$ such that $p^* : \HH^*(X,\Q) \to \HH^*(X,\Q)$ is an idempotent with image $H$.
	\end{conj}
	
	We note that the standard conjectures are implied by the (generalized) Hodge conjecture. Therefore, the generalized Hodge conjecture \ref{conj:GHC} for all complex smooth projective varieties implies the validity of Conjecture \ref{conj:strongGHC}. In particular, Conjectures \ref{conj:GHC} and \ref{conj:strongGHC} are equivalent when considered \emph{for all} complex smooth projective varieties.
	This stronger formulation of the generalized Hodge conjecture will be crucial to our main results\,; see Propositions~\ref{P:niveau} and~\ref{prop:projlef}.

	\section{Generically defined cycles}

	\subsection{Generically defined cycles on self-products of abelian varieties}
	A fundamental result of Grothendieck and Mumford \cite[Theorem 7.9]{mumford}
	is that, for $N\geq 3$, the fine moduli scheme $\mathcal{A}_{g,d,N}$ for
	polarized abelian varieties of degree $d^2$ and dimension $g$ with level-$N$
	structure exists, and that it is moreover quasi-projective over
	$\operatorname{Spec} \Z$.

	\begin{defn}[Generically defined cycles on abelian varieties] \label{D:unidef}
		Let $m,g$ and $d$ be positive integers. A \emph{generically defined cycle on
			the $m$-fold
			power
			of a polarized complex abelian
			variety} $A$ of degree $d^2$ and dimension $g$ is a cycle in $\CH^*(A^m)$ that
		is the
		restriction, for some integer $N\geq 3$, of a cycle on the $m$-fold power of
		the
		universal polarized abelian
		variety of degree $d^2$ and dimension $g$ with level-$N$ structure.
	\end{defn}
	
	For the sake of this paper we only consider cycles with rational coefficients,
	but of course the definition of generically defined cycles on abelian varieties
	makes sense for Chow groups with integral coefficients. However, with rational coefficients, the definition is independent of the choice of a level structure\,:

	\begin{rmk} \label{rmk:N}	
		By considering the natural finite \'etale morphism $\mathcal{A}_{g,d,M} \to
		\mathcal{A}_{g,d,N}$ for integers $M,N\geq 3$ such that $N$ divides $M$,
		we see that generically defined cycles on the $m$-fold power of a polarized
		complex abelian variety $A$ are in fact  the
		restriction, \emph{for all} integers $N\geq 3$, of a cycle on the $m$-fold
		power of the
		universal polarized abelian
		variety of degree $d^2$ and dimension $g$ with level-$N$ structure. In
		particular, generically defined cycles on the $m$-fold power of a polarized
		complex abelian variety $A$ form a $\Q$-sub-algebra of $\CH^*(A^m)$. 
	\end{rmk}

	\begin{rmk}[Universally defined cycles on abelian varieties]\label{R:some}
		In our applications, the generically defined cycles that we are going to
		consider will actually satisfy the following stronger condition.
		Let $m$ and $g$ be nonnegative integers. 
		A \emph{universally defined cycle} on the $m$-fold power of polarized abelian
		varieties of dimension $g$ consists,  for every polarized abelian scheme
		$\mathcal{A} \to B$ of relative dimension $g$ over a smooth quasi-projective
		complex variety $B$, of a
		cycle $z_\mathcal{A} \in \CH^*(\mathcal{A}^{m}_{/B})$ such that for every
		morphism $f:B'\to B$ of smooth quasi-projective complex varieties
		$z_{\mathcal{A}}$ restricts to $z_{\mathcal{A}\times_B B'}$ under the natural
		morphism $(\mathcal{A}\times_B B')^{m}_{/B'} \to \mathcal{A}^{m}_{/B}$. Here
		the
		abelian scheme $\mathcal{A}\times_B B' \to B'$ is understood to be equipped
		with
		the polarization induced by that of $\mathcal{A}$.
	\end{rmk}

	\begin{rmk}\label{R:uni}
		It is clear that, when restricted to the $mn$-fold powers of polarized abelian
		varieties
		of dimension $g$,  a cycle that is generically defined for $m$-fold powers of
		polarized
		abelian varieties of dimension $ng$ is generically defined for $mn$-fold
		powers of polarized abelian varieties of dimension $g$.
	\end{rmk}
	
	\begin{ex}\label{Ex:universal}
		The polarization of a polarized abelian variety is generically defined. 
		Likewise,  the
		first Chern class of
		the Poincar\'e line-bundle  (see \S \ref{S:def}) and the 
		correspondences $\Lambda_A^i$ of \eqref{E:Lambda} are generically defined on
		2-fold products of polarized abelian varieties. 
	\end{ex}

	For future use, let us give the following examples of generically defined
	self-correspondences on abelian varieties\,:
	
	\begin{lem}\label{L:kleiman}Suppose that $(A,L)$ is a polarized complex abelian
		variety
		of dimension $g$. Then there exist, for all integers $k$ and $n$, idempotent
		correspondences $p^{k,n} \in \DCH^g(A\times A)$ that are  generically defined
		for
		$2$-fold products of abelian varieties, and whose action in cohomology are the
		orthogonal projectors $$p^{k,n} : \HH^*(A,\Q) \to
		L^n\HH^{k-2n}(A,\Q)_{\mathrm{prim}} \to \HH^*(A,\Q).$$ 
		In particular, the Chow--K\"unneth projectors $\pi_A^k := \sum_n p^{k,n}$ are
		generically defined.
	\end{lem}
	\begin{proof}
		Kleiman \cite[Proposition 2.3]{kleiman} showed that the orthogonal projectors
		$p^{k,n}$ are algebraic for all smooth projective complex varieties that
		satisfy
		Grothendieck's Lefschetz standard conjecture. In fact, given a polarized
		abelian
		variety $(A,L)$ it is shown in \cite[Proposition 1.4.4]{kleiman} that the
		projectors $p^{k,n}$ are the classes of cycles (denoted abusively also
		$p^{k,n}$) that  belong to the sub-algebra of $\CH^*(A\times A)$ generated by
		the $\Lambda_A^i$ for $0\leq i \leq 2g$ (see \eqref{E:Lambda}). Since the
		cycles $\Lambda^i_A$ are
		generically defined for $2$-fold products of abelian varieties (Example
		\ref{Ex:universal}), so are the cycles $p^{k,n}$. Finally, note that the
		cycles
		$\Lambda_A^i$ belong to $\DCH^*(A\times A)$ by O'Sullivan's
		Theorem~\ref{T:Osullivan} so
		that the cycles $p^{k,n}$ belong to  $\DCH^{g}(A\times A)$\,; these are
		idempotents by O'Sullivan's Theorem.
	\end{proof}
	\begin{rmk}\label{R:symm}
		The cycles $p^{k,n}$ can be defined explicitly in terms of the $\Lambda_A^i$
		by
		carrying cohomological computations similar to \cite[Proposition
		1]{beauvilleFourier} or \cite[Proposition 7.3]{sv} (note that in
		\cite[Proposition 7.3]{sv} there is a sign error\,: $(-1)^i$ should read
		$(-1)^{i+g}$). Moreover, since the Chow--K\"unneth projectors $\pi_A^k := \sum_n
		p^{k,n}$ of Lemma~\ref{L:kleiman} are symmetrically distinguished, they coincide with the
		ones
		of Deninger--Murre \cite{dm}. In particular, writing $\mathfrak{h}^k(A)$ for
		the
		direct summand of the Chow motive $\mathfrak{h}(A)$ corresponding to the
		Chow--K\"unneth projector $\pi_A^k$, we have the Beauville
		decomposition~\cite{beauville}\,:
		\begin{equation}\label{E:beauville}
			\CH^i(A)_{(j)} :=
			\CH^i(\mathfrak{h}^{i-2j}(A)) = \{a \in \CH^i(A) : [n]^*a = n^{i-2j}a \ \text{for all}\ n\in \Z \}.
		\end{equation}
		Here, $[n] : A \to A$ is the multiplication-by-$n$ homomorphism.
	\end{rmk}
	
	\begin{ques}[generically defined cycles and symmetrically distinguished
		cycles]\label{R:gdsd}
		It is tempting to ask whether generically defined cycles on powers of
		abelian
		varieties are symmetrically distinguished in the sense of O'Sullivan
		\cite{o'sullivan}, in particular whether generically defined cycles are invariant under the multiplication by $-1$ homomorphism. (All the explicit cycles that we consider that are
		generically defined are also symmetrically distinguished.) Since the
		$\Q$-sub-algebra
		of $\CH^*(A^m)$ consisting of
		symmetrically distinguished cycles injects in cohomology, and since Hodge
		classes on $A^m$ consist of polynomials in $p_i^*L$ and
		$p_{i,j}^*c_1(\mathcal{P_A})$ for $A$ very general (see Theorem~\ref{T:gen2}),
		this would imply that generically defined cycles on $m$-fold powers of abelian
		varieties
		are polynomials in $p_i^*L$ and $p_{i,j}^*c_1(\mathcal{P_A})$\,; see also
		Proposition~\ref{prop:R}(a) below. 
		This would
		constitute a generalization (with rational coefficients) of the Franchetta
		conjecture for abelian varieties\,; see the recent \cite{fp} where it is shown in particular
		that a generically defined cycle (with rational coefficients) of codimension
		$1$
		on polarized abelian varieties is a rational multiple of the polarization.  
		
		Given the fact that a general complex principally polarized abelian threefold is isomorphic to the Jacobian of a smooth projective curve of genus, one could be led to think that the Ceresa cycle (which for a very general such abelian threefold is not symmetrically distinguished) provides a  generically defined cycle for principally polarized threefolds. This is however not the case. Indeed, the morphism $\mathcal{C}_{3,N} \to \mathcal{A}_{3,N}$ from the moduli space of genus 3 curves with level $N$ structure to the moduli space of principally polarized abelian threefolds with level $N$ structure ($N\geq 3$) is a degree 2 morphism, due to the fact that a general curve of genus 3 has no non-trivial automorphism whereas an abelian variety always admits an involution. Since the Ceresa cycle is sent to minus itself under the multiplication by $-1$ homomorphism, we see that the Ceresa cycle is in fact fiberwise zero over $\mathcal{A}_{3,N}$. We refer to \cite{nori} for more details.
	\end{ques}
	
	\begin{ques}[generically defined cycles on hyperK\"ahler varieties]
		It is also tempting to ask whether the sub-ring of the Chow ring
		consisting
		of generically defined cycles on polarized hyperK\"ahler
		varieties of a fixed deformation type injects into cohomology\,; see~\cite{flv}
		for precise statements and some evidence. Note that contrary to the case of abelian varieties, we do
		not expect generically defined cycles to be sums of intersections of
		divisors or even Chern classes\,;
		for instance,
		for
		hyperK\"ahler varieties that are deformations of $\mathrm{Hilb}^n(K3)$, the
		Beauville--Bogomolov--Fujiki
		class defines a generically defined Hodge class on the $2$-fold product, and
		we
		expect the existence of a generically defined cycle $L$ in $2$-fold powers of
		such varieties whose cohomology class is the Beauville--Bogomolov--Fujiki
		class\,; see \cite{sv}. 
	\end{ques}

	\subsection{The generalized Hodge conjecture for very general abelian
		varieties} We recall the well-known fact that for a very general abelian
	variety the Hodge coniveau filtration coincides with the primitive filtration.

	\begin{defn} Let $(X,L)$ be a smooth projective complex variety of dimension
		$d$, equipped with a polarization $L$.
		The \emph{primitive filtration} (with respect to $L$) is 
		$$\mathrm{P}^j\HH^k(X,\Q) = \bigoplus_{r\geq j} L^r
		\HH^{k-2r}(X,\Q)_{\mathrm{prim}},$$
		where $\HH^{i}(X,\Q)_{\mathrm{prim}} = \ker \left( L^{d-i+1} : \HH^i(X,\Q) \to
		\HH^{2d-i+2}(X,\Q) \right)$ for $i\leq d$, and is $0$ for $i>d$.
	\end{defn}
	Note that when $A$ is a very general abelian variety, there is up to scalar
	only
	one symmetric ample divisor on $A$. In particular, in this case, the primitive
	filtration
	does
	not depend on the choice of a polarization. 
	The following theorem is folklore.
	\begin{thm}[Generalized Hodge conjecture for very general abelian
		varieties]\label{T:gen}
		Let $A$ be a very general polarized complex abelian variety. Then 
		$$ \mathrm{P}^\ast \HH^k(A,\Q) = \mathrm{N}_H^\ast \HH^k(A,\Q) $$
		for all $k \geq 0$.
	\end{thm}
	\begin{proof}
		Since $A$ is very general, its Hodge group is dense in the
		symplectic group $\operatorname{Sp}(\HH^1(A,\Q))$. The proof thus reduces to a
		representation-theoretic argument. We refer to Hain's argument in \cite[Prop.
		4.4]{friedlander}, or to \cite[p. 135]{hazama}.
	\end{proof}
	
	\subsection{The generalized Hodge conjecture for self-powers of  very general
		abelian
		varieties}\label{S:gHpowers}
	A crucial step towards the proof of Theorem~\ref{T:main} is the following
	generalization due to Hazama \cite{hazama} of Theorem~\ref{T:gen} to
	self-powers of $A$.
	
	\begin{thm}[Hazama]\label{T:gen2}
		Let $A$ be a very general polarized complex abelian variety. Then,
		denoting $\iota_\Delta : A^m \to A^m \times A^m$ the diagonal embedding, we
		have
		$$ \mathrm{N}_H^n \HH^k(A^m,\Q) = \sum_Q
		((\iota_\Delta)_*Q)_*\HH^{k-2n}(A^m,\Q)$$
		for all $m,k \geq 0$, where the sum runs through all cycles $Q \in
		\CH^{n}(A^m)$ which are products of cycles of the form $(p_{i})^*L$,
		$(p_{i,j})^*P$.
		Here
		$p_{i} : A^m\to A$ and $p_{i,j} : A^m \to A^2$ are the
		natural projections, and $P \in \CH^1(A\times A)$ is the first Chern class of
		the Poincar\'e line-bundle (see \S \ref{S:def}).
	\end{thm}
	\begin{proof}
		This is due to Hazama \cite[Th. 5.1]{hazama}. (Note that a very general
		abelian
		variety is such that $\End^0(A) = \Q$ (hence of type I), and is such that its
		Hodge group coincides with its Lefschetz group (and hence \emph{stably
			nondegenerate} in the terminology of~\cite{hazama})). The proof is
		representation-theoretic and involves understanding the irreducible
		representations of $\mathrm{Sp}(H^1(A,\Q))$ that appear as direct summands of
		the representations $\bigwedge^{k_1} H^1(A,\Q) \otimes \cdots \otimes
		\bigwedge^{k_r} H^1(A,\Q)$ with $k_1+\cdots +k_r = k$. For a proof, we also
		refer to Theorem~\ref{thm:lefhodge}, where we will generalize Hazama's
		theorem.
	\end{proof}
	
	As a consequence, we can prove (a finer version of) Conjecture~\ref{conj:strongGHC} for powers of a very general abelian variety\,:
	
	\begin{prop}\label{P:niveau}
		Let $A$ be a very general polarized complex abelian variety of dimension $g$,
		and let
		$m$
		be an integer. Then for every integers $k$ and $n$ there exists an idempotent
		correspondence $q^{k,n} \in \CH^{gm}(A^m\times A^m)$ inducing the  projection
		$\HH^*(A^m,\Q) \to \mathrm{N}_H^n\HH^k(A^m,\Q) \to \HH^*(A^m,\Q)$, which is a
		linear combination of correspondences of the form $$\mathfrak{h}(A^m)
		\stackrel{\rho}{\longrightarrow} \mathfrak{h}(A^m)(n) 
		\stackrel{\zeta}{\longrightarrow} \mathfrak{h}(A^m),$$ where $\rho$ and
		$\zeta$
		are both symmetrically distinguished cycles and
		generically defined cycles on
		$2m$-fold powers of abelian
		varieties of dimension $g$.  Moreover,
		such a correspondence is unique modulo homological equivalence.
	\end{prop}
	\begin{proof}
		By Theorem~\ref{T:gen2}, we have 
		\begin{equation*}
			\mathrm{N}_H^n\HH^k(A^m,\Q) = \Gamma_*\HH^{k-2n}(B,\Q),
		\end{equation*} where $B := \coprod_Q A^m$ is the disjoint union of copies of
		$A^m$ indexed by the correspondences $Q$, and 
		$\Gamma := \sum_Q (\iota_\Delta)_*Q \in \CH^{gm+n}\left(B \times A^m  
		\right).$ Since the correspondences $Q$ are symmetrically distinguished and
		generically defined for $2m$-fold
		products of abelian varieties, the correspondence $\Gamma$ is symmetrically
		distinguished and generically
		defined for $2m$-fold products of abelian varieties. We view $\Gamma$ as a
		morphism of Chow motives $\mathfrak{h}(B)(n)  \rightarrow \mathfrak{h}(A^m)$.
		In
		the proof below, we
		are
		going to construct  idempotent correspondences $q^{k,n}$, with the
		factorization property stated in the proposition,  whose action on cohomology
		is
		the orthogonal projector on 
		$\Gamma_*\HH^{k-2n}(B,\Q)$, for all abelian varieties $A$ (the hypothesis that
		$A$ is very general is only used to compare $\Gamma_*\HH^{k-2n}(B,\Q)$  with
		$\mathrm{N}_H^n\HH^k(A^m,\Q)$\,; these coincide when $A$ is very general by
		Theorem~\ref{T:gen2}).
		
		By Lemma~\ref{L:kleiman}, the endomorphisms $p^{j,r} \in
		\mathrm{End}(\HH^*(A',\Q))$ are induced by cycles that belong to
		$\DCH^{mg}(A'\times A')$ and are generically defined on
		$2$-fold products of abelian varieties $A'$ of dimension $mg$. Restricting to
		$2m$-fold products of abelian varieties of dimension $g$, we see by Remark
		\ref{R:uni} that the $p^{j,r} \in \mathrm{End}(\HH^*(A^m,\Q))$ are in fact
		induced by generically
		defined cycles on $2m$-fold products of abelian varieties of dimension $g$.
		Hence, the endomorphisms
		$s_j:= \sum_r (-1)^rp^{j,r} \in \mathrm{End}(\mathfrak{h}(A^m))$ and the
		Chow--K\"unneth
		projectors $\pi^{j} := \sum_r p^{j,r}\in \mathrm{End}(\mathfrak{h}(A^m))$ are
		cycles that belong to $\DCH^{mg}(A^m\times A^m)$  and are generically
		defined on $2m$-fold products of abelian varieties of dimension $g$.
		
		Denote $s:= \bigoplus_Q s_{k-2n} \in \mathrm{End}(\mathfrak{h}(B))$ and
		$\Lambda_B =
		\coprod_Q \Lambda_{A^m}^{2mg+2n-k}$.
		Since the Hodge structure $\HH^{2gm-k+2n}(B,\Q)$ ($=\bigoplus_Q
		\HH^{2gm-k+2n}(A^m,\Q)$), equipped with the pairing $\alpha \otimes \phi
		\mapsto \int_B \alpha \cup (s\circ \Lambda_B)_*\phi$, is polarized, we have
		(see e.g. \cite[Lemma 1.6]{vialniveau})
		$$\mathrm{im} \left( \left( \Gamma \circ s \circ \Lambda_B \circ {}^t\Gamma
		\circ \Lambda_{A^m}^k\right)_*\right) = \mathrm{im}\left(\Gamma_*\right) =
		\mathrm{N}_H^n\HH^k(A^m,\Q) .$$ 
		Moreover the correspondence $s \circ \Lambda_B \circ {}^t\Gamma \circ
		\Lambda_{A^m}^k$ acts as
		zero on the orthogonal complement of $\Gamma_*\HH^{k-2n}(B,\Q)$.
		By the theorem of Cayley--Hamilton, we may thus express the orthogonal
		projector
		on $\Gamma_*\HH^{k-2n}(B,\Q)$ as a polynomial (with zero constant term) in the
		endomorphism $ \left( \Gamma \circ s \circ
		\Lambda_B \circ {}^t\Gamma \circ \Lambda_{A^m}^{k}\right)_* \in
		\mathrm{End}(\HH(A^m,\Q))$. This shows that the orthogonal projector
		on $\Gamma_*\HH^{k-2n}(B,\Q)$ is induced by a cycle that is a linear
		combination of cycles with the factorization property stated in the
		proposition.
		
		Finally, concerning the uniqueness of $q^{k,n}$ modulo homological
		equivalence, let us prove more generally that an endomorphism of a Hodge
		structure
		$H$ of weight $k$, with image $\mathrm{N}_H^nH$ is unique. Assume $q$ and $q'$
		are two such endomorphisms. By definition of the Hodge coniveau filtration,
		$q$
		and $\mathrm{id}_H - q'$ are mutually orthogonal projectors. Therefore $q$ and
		$q'$ commute\,; we conclude by using the elementary fact that
		two idempotent endomorphisms of a vector space coincide when they commute with
		one another and have the same image.
	\end{proof}

	\begin{rmk}[Refined Chow--K\"unneth decompositions]\label{rmk:refined}
		Without going into the details, we simply note that Proposition~\ref{P:niveau}
	shows that the refined
		Chow--K\"unneth projectors of \cite{vialniveau} can be constructed
		unconditionally for the powers of a very general abelian variety.
		Due to Proposition~\ref{prop:projlef}
	and Corollary~\ref{cor:GHCsurface} below, the same holds for self-powers of
		elliptic curves or abelian surfaces. 
		In particular, since generalized Kummer varieties are motivated by an abelian surface (see \S \ref{sec:motab}),
		they
		admit a refined Chow--K\"unneth
		decomposition in the sense of \cite{vialniveau}.
	\end{rmk}

	\subsection{Generically defined cycles and the generalized Bloch conjecture}
Our main result concerning generically defined cycles is the
	following  slight generalization of
	Theorem~\ref{T:main}\,:
	\begin{thm}\label{T:main2}
		Let $\gamma \in \CH^{*}(A^l\times A^m)$ be a generically defined cycle on
		the
		$(l+m)$-fold power of a polarized complex abelian variety $A$ of dimension
		$g$. 
		Assume
		that $\gamma^*\HH^*(A^m,\Q) \subseteq \mathrm{N}_H^n \HH^*(A^l,\Q)$
		for some (equivalently, for all) polarized complex abelian variety $A$ of
		dimension $g$ and degree $d^2$.
		We have\,:
		
		\begin{enumerate}
			\item \label{I:1} If $l=m$ and  $\gamma \in \CH^{mg}(A^m\times A^m)$, 	then
			$\gamma_*$ acts nilpotently on $\CH_r(A^m) $ for all $r<n$. In particular, if
			$\gamma$ is an idempotent correspondence, then $\gamma_*\CH_r(A^m)=0 $ for
			all
			$r<n$.
			\item \label{I:2} If $\gamma$ is a symmetrically distinguished
			correspondence,
			then $\gamma_*\CH_r(A^m)=0 $ for all $r<n$.
		\end{enumerate}
	\end{thm}	
	
	\begin{proof}
		The notations are those of Proposition~\ref{P:niveau} and its proof. First
		assume that $A$ is very general.
		By assumption,  ${}^t\gamma$ has same homology class as  $ \sum_k q^{k,n}\circ
		{}^t\gamma$. In particular, after transposing the above equality, $\gamma$ is
		a linear combination of morphisms that factor
		through the morphism of homological motives ${}^t\Gamma :
		\mathfrak{h}^{\mathrm{hom}}(A^m) \to \mathfrak{h}^{\mathrm{hom}}(B)(-n) =
		\bigoplus_Q \mathfrak{h}^{\mathrm{hom}}(A^m)(-n)$. 
		Since all the cycles considered in the proof of Proposition
		\ref{P:niveau}
		are generically defined for $2m$-fold powers of abelian varieties of dimension
		$g$, the above conclusion in fact holds without assuming that the abelian 
		variety $A$ is very general.

		\eqref{I:1}	We are assuming that $\gamma$ is a self-correspondence on $A^m$ of
		degree $0$\,; \emph{i.e.}, that it is a morphism $\mathfrak{h}(A^m) \to
		\mathfrak{h}(A^m)$. By finite-dimensionality of the motive of abelian
		varieties
		\cite{kimura},
		some power of $\gamma$, say $\gamma^{\circ N}$, factors through the morphism
		of Chow motives  ${}^t\Gamma :
		\mathfrak{h}(A^m) \to \mathfrak{h}(B)(-n)$. Therefore the action of
		$\gamma^{\circ N}$
		on
		$\CH_r(A^m)$ factors through $\CH_{r-n}(B)$\,; the group $\CH_{r-n}(B)$ is
		obviously zero for $r < n$. 
		
		\eqref{I:2}	 Finally, in order to see that $\gamma_*\CH_r(A^m)=0 $ for all
		$r<n$  if $\gamma$ is assumed to be symmetrically distinguished, it suffices
		to
		note that all the cycles appearing in Proposition~\ref{P:niveau} and its proof
		are symmetrically distinguished, so that  ${}^t\gamma$ is equal to $ \sum_k
		q^{k,n}\circ
		{}^t\gamma$ modulo rational equivalence, and hence is a linear combination of
		morphisms that factor
		through the morphism of motives ${}^t\Gamma :
		\mathfrak{h}(A^m) \to \mathfrak{h}(B)(-n) =
		\bigoplus_Q \mathfrak{h}(A^m)(-n)$. 
	\end{proof}

	\section{Lefschetz representations}\label{sec:LR}
	
	\subsection{The Lefschetz group}\label{sec:lefschetz}
	In this paragraph, we fix definitions and notations as well as recall the basic properties of
	the Lefschetz group.
	A rational Hodge structure $H$ of pure weight $k$ can be described as a
	$\Q$-vector space of finite dimension with a homomorphism of
	$\mathds{R}$-groups
	$\mathrm{Res}_{\C/\mathds{R}} \mathds{G}_m \to
	\operatorname{GL}(H_{\mathds{R}})$, where $\mathrm{Res}$ denotes restriction of
	scalars \emph{\`a la} Weil.
	The \emph{Mumford--Tate group} $\mathrm{MT}(H)$ of a rational Hodge structure
	$H$ is the $\Q$-Zariski closure of the image of this homomorphism\,; it is 
	connected. If $H$ is of pure weight, we will be interested in a
	smaller group called the \emph{Hodge~group} of $H$\,: $\mathrm{Hdg}(H)$ is the
	$\Q$-Zariski closure of the image of the circle group
	$\mathrm{Ker}\left(\mathrm{Res}_{\C/\mathds{R}} \mathds{G}_m \to \mathds{G}_m
	\right) \to \operatorname{GL}(H_{\mathds{R}})$. (Concretely, a Hodge structure
	$H$ of weight $k$ determines a homomorphism	$\lambda_H : S^1 \to
	\operatorname{GL}(H_\mathds{R})$ where $S^1$ is the unit circle in the complex
	plane,
	such that $\lambda_H(z)$ acts on $H^{p,q}$ by multiplication by $z^{p-q}$, and
	conversely such a homomorphism induces an eigenspace decomposition of $H_\C$
	that satisfies the Hodge symmetry $\overline{H^{p,q}} = H^{q,p}$). The Hodge
	group is a connected  group characterized  by the property that its invariants
	in $H^m \otimes (H^\vee)^n$ are precisely the Hodge classes
	for all non-negative integers $m$ and $n$.
	Moreover the Hodge sub-structures of $H^m \otimes (H^\vee)^n$ are precisely the
	sub-representations of the Hodge group of $H$.
	The element $C := \lambda_H(i)$ is called the \emph{Weil operator}. A
	\emph{polarization} of $H$ is a morphism of Hodge structures $\phi :
	H\otimes_\Q
	H \to \Q(-k)$ such that $\phi(x,Cy)$ is symmetric and positive definite on
	$H_\mathds{R}$. When $H$ admits a polarization, then its Mumford--Tate and Hodge groups are reductive.

	Let $A$ be a complex abelian variety.
	Cup-product defines an isomorphism of graded $\Q$-algebras
	$\bigwedge^*\HH^1(A,\Q) \to \HH^*(A)$. Via the isomorphism 
	\begin{equation}\label{eq:pairing}
		\HH^2(A,\Q) \simeq \Hom (\bigwedge^2V(A), \Q(1)),
	\end{equation}
	the cohomology class of a divisor $D$ on $A$ defines a skew-symmetric pairing
	$\phi_D : V(A) \times V(A) \to \Q(1)$, where $V(A) := \HH_1(A,\Q)$. When $D$ is ample, $\phi_D$ is
	non-degenerate and defines a polarization on the $\Q$-Hodge structure $V(A)$.
	We
	let $\rho_D$ denote the involution of the $\Q$-algebra $\End_\Q(V(A))$, which
	to
	an endomorphism of $V(A)$ associates its adjoint with respect to $\phi_D$\,;
	its
	restriction to $$\End^0(A) := \End(A)\otimes_\Z \Q$$ is the \emph{Rosati
		involution} defined by $D$. 	By definition, the Hodge group $\mathrm{Hdg}(A)$
	of
	a complex abelian variety $A$ is the Hodge group attached to the polarized
	$\Q$-Hodge structure $$V(A) := \HH_1(A,\Q).$$ 
	Due to the semi-simplicity of the category of polarized
	$\Q$-Hodge structures, the Mumford--Tate group and the
	Hodge group of a polarized Hodge structure are reductive groups.
	
	\begin{defn}
		For a complex abelian variety $A$ endowed with a polarization $L$, 
		the \emph{Lefschetz group} $L(A)$ is defined 
		to be the algebraic subgroup of $\operatorname{GL}(V(A))$ such that,
		for all commutative $\Q$-algebras $R$,
		$$L(A)(R) = \{\gamma \in C(A)\otimes_\Q R : \rho_L(\gamma)\gamma = 1\}.$$
		Here $C(A)$ is the centralizer of $\End^0(A)$ in $\End_\Q (V(A))$.
		The Lefschetz group can also be viewed as the centralizer of
		$\End^0(A)$ in $\operatorname{Sp}(V(A),\phi_L)$.
	\end{defn}
	
	The Lefschetz group does not depend of the choice of a polarization\,: given
	any two ample line-bundles $\mathcal{L}$ and $\mathcal{L}'$, there is an
	element
	$\eta \in \End^0(A)$ and a positive integer $m$ such that $m\phi_{\mathcal{L}}
	=
	\phi_{\mathcal{L}'}\eta$. In what follows, the polarization will usually be
	understood from the context, and we will therefore write simply $\rho$ for the
	Rosati involution, and 
	$\phi$ for the skew-symmetric form. In
	general, we have the inclusions
	$$\mathrm{Hdg}(A) \subseteq L(A) \subseteq \operatorname{Sp}(V(A),\phi).$$
	The Lefschetz group of $A$ naturally acts on the $\Q$-vector spaces
	$V(A)^{\otimes n} \otimes (V(A)^\vee)^{\otimes m}$, and we will refer to these
	as \emph{Lefschetz representations}.
	While the Hodge group doesn't behave well with respect to products, the
	Lefschetz group enjoys the following property\,:
	\begin{lem}[Murty \cite{murty}] \label{lem:murty}
		 If $A$ is isogenous to a product
		$A_1^{m_1}\times \cdots \times A_s^{m_s}$, with the $A_i$ simple and pairwise
		non-isogenous,
		then 
		$$L(A) \cong L(A_1)\times \cdots\times L(A_s)$$ with the factor $L(A_i)$
		acting
		diagonally on $H_1(A_i,\Q)^{\oplus m_i}$ and acting as zero on
		$H_1(A_j,\Q)^{\oplus m_j}$ for $j\neq i$.
	\end{lem}

	Recall that $\R^*(A) \subset \CH^*(A)$ denotes the $\Q$-sub-algebra generated
	by
	symmetric divisors and that it maps isomorphically onto its image
	$\overline{\R}^*(A)$  in $\HH^*(A,\Q)$ via the cycle class map by O'Sullivan's
	Theorem~\ref{T:Osullivan}. 
	The statement of the following theorem is taken from Milne \cite[Thm.~3.2]{milne} where it is proved more generally for abelian varieties defined over any algebraically closed fields, but its origin
	can be
	traced back to work of Tankeev \cite{tankeev}, Ribet~\cite{ribet},
	Murty~\cite{murty}, and Hazama \cite{hazama2}.
	
	\begin{thm}\label{thm:lefinv}
		The \emph{Lefschetz group} $L(A)$ of a complex abelian variety $A$ is such
		that
		$\overline{\R}^s(A^r) = \HH^{2s}(A^r,\Q)^{L(A)}$ inside $\HH^{2s}(A^r,\Q)$ for
		all non-negative integers~$r$ and~$s$.\qed
	\end{thm}
	
	Since the Hodge classes in $\HH^{2s}(A^r,\Q)$ are precisely the invariant
	classes under the action of the Hodge group $\mathrm{Hdg}(A)$, it follows from
	the Lefschetz $(1,1)$-theorem that the Hodge conjecture holds for powers of
	abelian varieties for which the inclusion $\mathrm{Hdg}(A) \subseteq L(A)$ is
	an
	equality. This is for example the case for elliptic curves, and abelian
	varieties of prime dimension\,; see \cite{tankeev} and \cite{ribet}.

	\subsection{Lefschetz groups and the Albert	classification}\label{sec:albertlef}
	Let $A$ be a complex abelian variety. 
	The proof of Theorem~\ref{thm:lefinv} proceeds through the computation of the
	Lefschetz group $L(A)$. 
	We start this paragraph by reviewing how the Lefschetz group of a simple complex abelian variety can be computed  via the characterization of the
	possible algebras $\End^0(A)$\,; for this we follow Murty \cite{murty}, and we refer to Shimura \cite{shimura} for the
	classification of such algebras via the Albert classification of division
	algebras with a positive involution.
	Set $D:= \End^0(A) = \End(A)\otimes_\Z
	\Q$. The Rosati involution $\rho$ of the
	semi-simple $\Q$-algebra $D$ induced by a polarization of $A$  defines a
	\emph{positive} involution of $D$ in the sense that $D$ has finite dimension
	over $\Q$ and the reduced trace $\mathrm{tr}_{D/\Q}(x\rho(x))$ is positive for
	all non-zero $x\in D$.
	From now on,
	we assume that $A$ is simple\,; in that case, $D$ is a division algebra. The
	involution $\rho$ restricts to a positive involution of the center $Z$ of $D$,
	and we
	denote $F$ the set of elements $z\in Z$ such that $\rho(z) = z$. As we have
	$\operatorname{tr}(z^2)>0$ for every non-zero element $z$ of $F$, the field $F$
	must be a totally real field. We set $d:= [D:Z]^{1/2}$ and $f:= [F:\Q]$. 
	According to Albert the following possibilities can occur for division algebras
	endowed with a positive involution\,:
	
	\begin{enumerate}[Type I.]
		\item $D=F$ is a totally real field\,;
		\item $D$ is a central division algebra over $F$ such that $D\otimes_\Q
		\mathds{R}$ is isomorphic to the product of $f$ copies of the matrix algebra
		$\operatorname{M}_2(\mathds{R})$\,;
		\item $D$ is a central division algebra over $F$ such that $D\otimes_\Q
		\mathds{R}$ is isomorphic to the product of $f$ copies of the quaternion
		algebra
		$\mathds{H}$\,;
		\item $D$ is a central  division algebra over a totally imaginary quadratic
		extension $F_0$ of $F$. 
	\end{enumerate}
	Accordingly a simple abelian variety $A$ is said to have type I, II, III, or
	IV,
	if $D= \End^0(A)$ has type I, II, III, or IV, respectively. For endomorphism
	rings of simple complex abelian varieties of dimension $g$, there are further
	dimension restrictions on the division algebras, coming from that fact that $D$
	acts faithfully on the $2g$-dimensional vector space $V(A) = \HH_1(A,\Q)$ (and
	the fact
	that the action of $D$ commutes with the complex structure for type I), namely 
	$f | g$ for type I,  $2f | g$ for types II and III, and $fd^2 | 2g$ for type
	IV.
	Shimura \cite{shimura} showed that every division algebra with a positive
	involution occurs as the endomorphism algebra of a simple complex abelian
	variety, except in 5 exceptional cases\,; in particular
	\cite[Prop.~15]{shimura}, for a simple abelian variety of type III, $2f$ must
	divide $g$ strictly.
	
	We now fix a skew-symmetric non-degenerate pairing $\phi : V(A) \times V(A) \to
	\Q$ determined by a polarization of $A$, via \eqref{eq:pairing}.
	For each type, there exist a unique non-degenerate $F$-bilinear
	form\footnote{The trace pairing $D\times D \to \Q, (a,b) \mapsto
		\mathrm{tr}_{D/\Q}(ab)$ is
		non-degenerate, and 
		$B(x,y)$ is the unique element in $D$ satisfying $\mathrm{tr}_{D/\Q}(aB(x,y))
		= \phi(ax,y)$ for all $a \in D$.} 
	$$B : V(A) \times V(A) \to D$$
	such that $\phi(x,y) = \mathrm{tr}_{D/\Q} B (x, y)$, $B (ax, by) = aB (x, y)b$,
	and $B (y, x) = - \rho\circ B (x, y)$ for all $x, y \in V(A)$ and all $a, b
	\in D$. The Lefschetz group is then the restriction of
	scalars, from $F$ to $\Q$, of the unitary group of $B$\,:
	$$ L(A) = \mathrm{Res}_{F/\Q} \operatorname{U}(B) =
	\mathrm{Res}_{F/\Q}\operatorname{Aut}_D(V(A),B).$$
	Let $S$ be the set of embeddings of $F$ into $\mathds{R}$. We can then write
	$$
	V(A)_{\mathds{R}} = \bigoplus_{\lambda \in S}V_\lambda,$$
	where $V_{\lambda}:= V(A) \otimes_{F,\lambda}\mathds{R} = \{v \in
	V(A)_{\mathds{R}} : f(v) = \lambda(f) v \ \mbox{for all } f\in F\}$ is a real
	vector space
	of dimension $2g/f$. In fact, since $D$ commutes with
	$\mathrm{Hdg}(A)_\mathds{R}$, $V_{\lambda}$ is a real Hodge sub-structure of
	$V(A)_\mathds{R}$. Since $L(A)$ commutes with the action of $F$, we have
	$L(A)_{\mathds{R}} \subseteq \prod \operatorname{GL}(V_\lambda)$ and thus
	$$L(A)_{\mathds{R}} = \prod_{\lambda \in S}L_\lambda,\quad \text{with }
	L_\lambda = \operatorname{Aut}_{D_\lambda}(V_\lambda,B_\lambda),$$
	where	$L_\lambda$ acts trivially on $V_{\lambda'}$ unless
	$\lambda=\lambda'$. Here $D_\lambda := D \otimes_{F,\lambda} \mathds{R}$ and
	$B_\lambda$ is the non-degenerate real bilinear form that is the restriction of
	$B\otimes_{\Q} \mathds{R}$ to $V_\lambda
	\times V_\lambda$.

	For types II and III, there exists an $F$-basis $1,\alpha,\beta,\alpha\beta$
	for $D$, with $\alpha^2$ totally negative, $\beta^2$ totally positive for type
	II and totally negative for type III, and $\alpha\beta = - \beta\alpha$.
	Denoting $E:= F[\alpha]$, we have $D = E \oplus E\beta$, and we can write 
	$$B(x,y) = B_1(x,y) + B_2(x,y)\beta, \quad \text{with } B_1(x,y), B_2(x,y) \in
	E.$$
	Then $B_1 : V(A) \times V(A) \to E$ is a non-degenerate skew-Hermitian form,
	and $B_2 : V(A) \times V(A) \to E$ is a non-degenerate skew-symmetric form for
	type II and a non-degenerate symmetric form for type III.
	Given an embedding $\lambda : F\hookrightarrow \mathds{R}$, we denote
	$\sigma, \bar{\sigma} : E \hookrightarrow \C$ the conjugate extensions of
	$\lambda$ to $E$. We define $V_\sigma := V \otimes_{E,\sigma} \C$, and we
	remark that $B_{1,\sigma} := (B_1\otimes_\Q \C)|_{V_\sigma\times V_\sigma} =
	0$,
	while $B_{2,\sigma} := (B_2\otimes_\Q \C)|_{V_\sigma\times V_\sigma}$ is
	non-degenerate (and similarly with $\bar \sigma$ in place of $\sigma$). 
	\medskip
	
	\noindent The group
	$L_\lambda$ and its action on $V_\lambda$ are given as follows (see
	\cite{murty}, but also \cite{milne})\,:
	\begin{enumerate}[Type I.]
		\item	$L_\lambda = \operatorname{Sp}_{\frac{2g}{f}}(V_\lambda , B_\lambda
		)$ is a symplectic
		group acting via its standard representation on $V_\lambda$\,;
		\item	$L_{\lambda}\otimes_{\mathds{R}}\C =
		\operatorname{Sp}_{\frac{g}{f}}(V_{\sigma},B_{2,\sigma})$ is a symplectic
		group
		acting on $V_{\lambda}\otimes_{\mathds{R}} \C = V_{\sigma} \oplus
		V_{\bar{\sigma}}$ as one copy of the
		standard representation and one copy of its contragredient representation
		(which is isomorphic to the standard representation).
		\item	$L_{\lambda}\otimes_{\mathds{R}}\C  =
		\operatorname{O}_{\frac{g}{f}}(V_{\sigma},B_{2,\sigma})$ is an orthogonal
		group
		acting on $V_{\lambda}\otimes_{\mathds{R}} \C = V_{\sigma} \oplus
		V_{\bar{\sigma}}$ as one copy of the
		standard representation and one copy of its contragredient representation
		(which is isomorphic to the standard representation).\footnote{In that case
			$\frac{g}{f}$ is an even number $\geq 4$ by \cite[Prop.~15]{shimura}.}
		\item	$L_{\lambda}\otimes_{\mathds{R}}\C =
		\operatorname{GL}_{\frac{g}{df}} (\C)$ acts on
		$V_{\lambda}\otimes_{\mathds{R}} \C$ as the direct sum of the
		standard representation and its	contragredient representation.
	\end{enumerate}
	In particular, the Lefschetz group $L(A)$ is a reductive group.
	The Lefschetz group $L(A)$ was first computed by Ribet \cite{ribet} for type I
	and IV in the case $D=F$. It was computed in general by Murty
	\cite{murty}\footnote{For type III Murty finds that 
		$L_{\lambda}\otimes_{\mathds{R}}\C$ is
		a special orthogonal group\,; this is because, contrary to the convention we
		adopted here, he considers the connected
		component of the identity of the Lefschetz group.}.\medskip

	Here is a useful basic fact (which is made more precise in the proof of
	\cite[Th\'eor\`eme~6.1]{ancona}) that can be derived from the reductiveness of the Lefschetz group\,:
	\begin{lem}\label{lem:projlef}
		Let $A$ be a complex abelian variety, and let $H\subseteq \HH^*(A,\Q)$ be a
		Hodge sub-structure. Then  $H\subseteq \HH^*(A,\Q)$ is a
		$L(A)$-sub-representation if and only if 
there exists a projector
		 $\HH^*(A,\Q)  \to \HH^*(A,\Q)$ with image $H$ that is
		induced
		by an idempotent correspondence in $\R^{\dim A}(A\times A)$.
	\end{lem}
	\begin{proof} 
	Thanks to the fact that $L(A)$ is reductive,
$H$ is a $L(A)$-sub-representation of
$\HH^*(A,\Q)$ if and only if
there exists a $L(A)$-invariant projector $\HH^*(A,\Q)  \to \HH^*(A,\Q)$ with image~$H$.
By Theorem~\ref{thm:lefinv} a $L(A)$-invariant projector $\HH^*(A,\Q)  \to \HH^*(A,\Q)$ with image~$H$ is induced by an idempotent correspondence in $\overline{\R}^{\dim A}(A\times A)$, and by O'Sullivan's theorem~\ref{T:Osullivan} such an idempotent correspondence can be lifted to an idempotent correspondence in $\R^{\dim A}(A\times A)$.
	\end{proof}
	
	Finally, we observe that, since $D\otimes_{\Q} \C
	\cong \End_{\mathrm{Hdg}(A)_\C}(\HH_1(A,\C))$, the action of $D\otimes_{\Q} \C
	$
	on $V(A)_\C = H_1(A,\C)$ commutes with the Hodge decomposition. In particular,
	if $E$ is a field sitting inside $D$, the decomposition $V_\C = \oplus_{\sigma
		:
		E \hookrightarrow \C}V_\sigma$ is compatible with the Hodge decomposition, so
	that writing $V_\sigma^{1,0} = V^{1,0} \cap V_\sigma$ and similarly for
	$V_\sigma^{0,1}$, we have 
	\begin{equation}\label{eq:dec}
		V_\sigma = V_\sigma^{1,0}  \oplus V_\sigma^{0,1}.
	\end{equation}
	We note also that $$ \overline{V_\sigma^{1,0}} = 
	V_{\overline{\sigma}}^{0,1}.$$
	The following lemma  will be crucial to the proof
	of Theorem~\ref{thm:lefhodge}. Since it does not hold in general for simple
	abelian varieties of type IV (e.g. CM elliptic curves), our focus until \S
	\ref{s:GHC2} will be on abelian varieties of totally real type.

	\begin{lem}\label{lem:isotropic}
		Let $A$ be a simple complex abelian variety of type I, II or III. 
		Let $E$ be a maximal subfield of $\End^0(A)$, which we choose as above to be a
		CM field for types II and III. Let $V_\sigma := V(A) \otimes_{E,\sigma} \C$
		for
		an embedding $\sigma : E \hookrightarrow \C$.
		Then the decomposition \eqref{eq:dec} is a decomposition into isotropic
		subspaces for the non-degenerate form\footnote{Note that for types II and III,
			$B_\sigma = B_{2,\sigma}$.}
		$$B_\sigma := B_\C|_{ V_\sigma \times V_\sigma} : V_\sigma \times V_\sigma \to
		D\otimes_\Q \C.$$
		In particular, $V_\sigma$ is ``numerically Hodge symmetric'', meaning that
		$$\dim_\C V_\sigma^{1,0} = \dim_\C V_\sigma^{0,1}.$$
	\end{lem}	
	\begin{proof} Recall that, for $x,y \in V_\C$, $B_\C(x,y)$ is the unique
		element
		in $D\otimes_{\Q}\C$ such that 
		$$\mathrm{tr}_{D\otimes_{\Q}\C / \C} (aB_\C(x,y))  = \phi_{\C} (ax,y), \quad
		\text{for all } a \in D\otimes_{\Q}\C.$$
		Since $V^{1,0}$ and $V^{0,1}$ are isotropic subspaces for the form $\phi_\C$
		and since the action of $D\otimes_{\Q}\C$ on $V_\C$ is compatible with the
		Hodge
		decomposition,
		we deduce that $B_\C(x,y) = 0$ for all $x, y \in V^{1,0}$ (resp. for all $x,y
		\in V^{0,1}$). The lemma follows by restricting to the $\sigma$-component in
		the
		decomposition $V_\C = \oplus_{\sigma : E \hookrightarrow
			\C}V_\sigma.$
	\end{proof}
	
	In summary, for simple abelian varieties of totally real type,
	Lemma~\ref{lem:isotropic} provides the following relations between the Hodge
	decomposition and the decomposition of the Lefschetz group after base-change\,:
	\begin{enumerate}[Type I.]
		\item The Hodge decomposition 	$V_\lambda \otimes_{\mathds{R}} \C =
		V_\lambda^{1,0} \oplus V_\lambda^{0,1}$ is a decomposition into isotropic
		subspaces for the non-degenerate skew-symmetric form
		$\phi_{\lambda}\otimes_{\mathds{R}}\C$\,;	
		\item The Hodge decomposition $V_\sigma = V_\sigma^{1,0}  \oplus
		V_\sigma^{0,1}$  is a decomposition into isotropic subspaces for the
		non-degenerate skew-symmetric form $B_{2,\sigma}$\,;	
		\item The Hodge decomposition $V_\sigma = V_\sigma^{1,0}  \oplus
		V_\sigma^{0,1}$  is a decomposition into isotropic subspaces for the
		non-degenerate symmetric form $B_{2,\sigma}$.	
	\end{enumerate}

	\subsection{Around Weyl's construction} \label{s:weyl}
	Let $V$ denote the standard representation of one of the classical groups
	$\operatorname{Sp}_{2n}$ or $\operatorname{O}_{2n}$. 
	Precisely, given a basis $(e_1,\ldots, e_n, e_{-1},\ldots , e_{-n})$ of $V$, we
	will be interested in the
	representations of the following groups\,:
	\begin{enumerate}[(a)]
		\item $G= \operatorname{Sp}_{2n}(V,Q)$, where $Q$ is the skew-symmetric
		bilinear form dual to  $$\psi = \sum_{i=1}^n
		e_i\otimes e_{-i} - e_{-i}\otimes e_i \in V\otimes V.$$
		\item $G=\operatorname{O}_{2n}(V,Q)$, where $Q$ is the symmetric bilinear form
		dual to $$\psi = \sum_{i=1}^n
		e_i\otimes e_{-i} + e_{-i}\otimes e_i \in V\otimes V.$$
	\end{enumerate} 
	For each pair $I=\{p<q\}$ of integers between $1$ and $d$, the skew-symmetric
	form
	$Q$ (resp. the symmetric form $Q$) determines a contraction
	$$\Phi_I : V^{\otimes d} \to V^{\otimes (d-2)},$$
	$$v_1\otimes \cdots \otimes v_d \mapsto Q(v_p,v_q) v_1\otimes \cdots\otimes
	\hat{v}_p \otimes \cdots \otimes \hat{v}_q \otimes \cdots \otimes v_d,$$
	where a `hat' means that the term is omitted.
	Denote $V^{\langle d \rangle }$ the intersection of the kernels of all these
	contractions, \emph{i.e.}, 
	$$V^{\langle d \rangle } := \bigcap_I \ker (\Phi_I).$$
	We can also define 
	$$\Psi_I : V^{\otimes (d-2)} \to V^{\otimes d}$$
	by inserting $\psi$ in the $p,q$ factors. 
	We have a direct sum of $G$-representations (see \cite[Ex. 17.13]{FH} for the
	case $G = \operatorname{Sp}_{2n}$\,; the case $G = \operatorname{O}_{2n}$ is
	similar)
	\begin{equation}\label{eq:*}
		V^{\otimes d} = V^{\langle d \rangle } \oplus \sum_I \operatorname{im}\,
		(\Psi_I).
	\end{equation}
	
	By considering the action of the symmetric group $\mathfrak{S}_d$ on
	$V^{\otimes
		d}$, we also have a direct sum decomposition of
	$\mathfrak{S}_d$-representations
	\cite[Ex. 4.50]{FH}
	\begin{equation}\label{eq:**}
		V^{\otimes d} = \bigoplus_{\lambda \dashv d} \mathds{S}_\lambda V.
	\end{equation}
	Here the direct sum runs through all \emph{standard Young tableaux} in $d$
	entries, and $\mathds{S}_\lambda V$ is the Schur symmetrizer attached to the
	underlying Young diagram. Moreover, for each standard Young tableau $\lambda$,
	there is an idempotent $p_\lambda \in \Q[\mathfrak{S}_d]$ (a rational multiple
	of the Young symmetrizer $c_\lambda$) such that $\mathds{S}_\lambda V =
	p_\lambda \cdot V^{\otimes d}$, and these idempotents are mutually orthogonal
	meaning that $p_\lambda p_\mu = 0$ for two distinct standard Young tableaux
	$\lambda$ and $\mu$. 
	
	Clearly a permutation $\sigma \in \mathfrak{S}_d$ commutes with the
	decomposition \eqref{eq:*}, and hence so do the idempotents $p_\lambda$. It
	follows that $V^{\langle d \rangle}$ further decomposes into a direct sum of
	$G$-representations as
	\begin{equation}
		V^{\langle d \rangle} = \bigoplus_{\lambda \dashv d} \mathds{S}_{\langle
			\lambda \rangle }V, \quad \text{where} \ \mathds{S}_{\langle \lambda \rangle
		}V
		:= \mathds{S}_\lambda V\cap V^{\langle d \rangle }.
	\end{equation}

	In order to state the next proposition, we need to introduce some notations. We
	follow Bourbaki \cite[Chap. VIII, \S 13.3 \& \S 13.4]{bourbaki}. Let $E_{i,j}$
	be the $2n\times 2n$ matrix expressed in the basis $(e_1,\ldots, e_n,
	e_{-1},\ldots , e_{-n})$ whose entries are all zero except for the $(i,j)$-th
	entry which is $1$. For $1\leq i \leq n$, we define $$H_i :=  E_{i,i} -
	E_{-i,-i}$$ and we let $(\varepsilon_1, \ldots, \varepsilon_n)$ be the dual
	basis of $(H_1,\ldots, H_n)$.

	Given a standard Young tableau $\lambda$ on $d$ entries, we denote
	$(\lambda_1\geq \lambda_2\geq \cdots \geq \lambda_d)$ the underlying partition
	of $d$. Subsequently, the number $d = \sum_i\lambda_i$ will also be referred to
	as the \emph{length} of $\lambda$ and will be denoted $\ell(\lambda)$.

	\begin{prop} \label{prop:irred}
		Assume that $G$ is either $\operatorname{Sp}_{2n}$, or $\operatorname{O}_{2n}$
		with
		$n>1$.
		Then $ \mathds{S}_{\langle \lambda \rangle }V$ is an
		irreducible representation of $G$.
		\begin{enumerate}[(a)]
			\item	If $G= \operatorname{Sp}_{2n}$, then $ \mathds{S}_{\langle \lambda
				\rangle
			}V
			\neq  0$ if and only if $\lambda_{n+1} = 0$\,; in that case $
			\mathds{S}_{\langle \lambda \rangle }V$  is the irreducible representation of
			$\mathfrak{sp}_{2n}$ with highest weight $\lambda_1 \varepsilon_1 + \cdots +
			\lambda_n\varepsilon_n$.  
			
			\item  If $G=
			\operatorname{O}_{2n}$, then $ \mathds{S}_{\langle \lambda \rangle }
			V\neq  0$ if and only if the sum of the lengths of the first two columns of
			the Young tableau $\lambda$ is at most $2n$.\footnote{Representations
				of associated partitions restricted to $\operatorname{SO}_{2n}$ are
				isomorphic.
				Two
				partitions (each with the sum of the first two column lengths at most $2n$)
				are said to be
				\emph{associated} if the sum of the lengths of their first columns is $2n$
				and
				the other columns of their Young diagram have the same lengths.} 
			\begin{itemize}
				
				\item If  $\lambda = (\lambda_1\geq\cdots\geq
				\lambda_{n} = 0)$, then $\mathds{S}_{\langle \lambda \rangle }V$  is the
				irreducible representation of $\mathfrak{so}_{2n}$ with highest weight
				$\lambda_1 \varepsilon_1 + \cdots + \lambda_n\varepsilon_n$. 
				
				\item If  $\lambda =
				(\lambda_1\geq\cdots\geq
				\lambda_{n} > 0)$, then $\mathds{S}_{\langle \lambda \rangle }V$  is the
				direct sum
				of the
				two irreducible representations of $\mathfrak{so}_{2n}$ with highest weight
				$\lambda_1 \varepsilon_1 + \cdots +
				\lambda_{n-1}\varepsilon_{n-1}+ \lambda_n\varepsilon_n$, and $\lambda_1
				\varepsilon_1 + \cdots +
				\lambda_{n-1}\varepsilon_{n-1} - \lambda_n\varepsilon_n$.  
			\end{itemize}
		\end{enumerate}

	\end{prop}
	\begin{proof}
		If $G = \operatorname{Sp}_{2n}$, this is \cite[Thm.~17.11 and Cor.~17.21]{FH}.
		If $G = \operatorname{O}_{2n}$, this is \cite[Thm.~19.19 and Thm.~19.22]{FH}.
	\end{proof}

	\subsection{Lefschetz representations and the generalized Hodge conjecture for
		abelian varieties of totally real type}
	The generalized Hodge conjecture was established by Hazama~\cite{hazama} for 
	abelian
	varieties whose simple factors are of type I or II and whose Hodge group
	coincides with their Lefschetz group, for
	$n$-dimensional simple abelian varieties of type I with $n/e$ odd ($e = \dim_\Q
	\End^0(A)$) by Tankeev \cite{tankeev2} (and in particular for odd-dimensional
	simple abelian varieties of type I), for certain simple abelian varieties of
	CM-type by Tankeev \cite{tankeev2}. Abdulali~\cite{abdulali} and Hazama \cite{hazama3} showed that the
	generalized Hodge conjecture for  abelian varieties of CM-type is implied by
	the
	Hodge conjecture for the same class of abelian varieties.
	Here we take a different approach and establish a strong form of the
	generalized Hodge conjecture for Lefschetz sub-representations of abelian
	varieties of totally real type.

	\begin{thm}[GHC for Lefschetz sub-representations of abelian varieties of
		totally real type]\label{thm:lefhodge}
		Let $A$ be a complex abelian variety, and let $H \subseteq \HH^k(A,\Q)$ be a
		Lefschetz sub-representation of Hodge level $\leq k-2n$. Suppose that $A$ is
		of totally real type, \emph{i.e.}, that the
		simple factors of the isogeny
		class of $A$ have type I, II, or III.
		Then
		\begin{equation}\label{eq:inclusion}
			H\subseteq  \mathrm{Im} \left(\overline{\R}^n(A) \otimes \HH^{k-2n}(A,\Q)
			\stackrel{\cup}{\longrightarrow} \HH^k(A,\Q)\right).
		\end{equation}
	\end{thm}
	In fact, we are going to show a stronger statement, namely that the conclusion
	of Theorem~\ref{thm:lefhodge} holds, after tensoring with $\C$, for
	$L(A)_\C$-sub-representations of $\HH^k(A,\C)$\,; see \eqref{eq:strong}.
	
	The key point towards the proof of Theorem~\ref{thm:lefhodge} consists in
	computing the
	``Hodge level'' of the representations $\mathds{S}_{\langle \lambda \rangle }
	V$
	for
	$G = \operatorname{Sp}_{2n}$ or $\operatorname{O}_{2n}$. 
	Strictly speaking, the spaces $V$ we are going to deal with are not Hodge
	structures. Rather, as described in Section~\ref{sec:albertlef}, they are
	complex vector spaces $V_\sigma$ endowed with a basis  $(e_1,\ldots, e_n,
	e_{-1},\ldots , e_{-n})$ and a (skew-)symmetric form $\psi =  \sum_{i=1}^n
	(e_i\otimes
	e_{-i} \pm e_{-i}\otimes e_i)$, together with an action of
	$\operatorname{GL}_1$ given by $z\cdot
	e_i = ze_i$ and $z\cdot e_{-i} = z^{-1}e_{-i}$ for $1\leq i \leq n$.	
	Since the action of $\operatorname{GL}_1$ on $\psi$ is the identity and since
	it commutes with the action of permutations in $\mathfrak{S}_d$ on $V^{\otimes
		d}$, the decompositions \eqref{eq:*} and \eqref{eq:**} commute with the action
	of $\operatorname{GL}_1$. In particular, for a Young tableau $\lambda$ of
	length
	$d$, we have a decomposition 
	$$\mathds{S}_{\langle \lambda\rangle}V = \bigoplus_{p+q=d} (\mathds{S}_{\langle
		\lambda\rangle}V)^{p,q},$$ 
	where $$(\mathds{S}_{\langle \lambda\rangle}V)^{p,q} := \{w \in
	\mathds{S}_{\langle \lambda\rangle}V : z\cdot w = z^{p-q}w \ \mbox{for all } z
	\in \operatorname{GL}_1(\C)\}.$$
	
	Our Theorem \ref{thm:lefhodge} generalizes Hazama's \cite[Theorem~5.1]{hazama}
	by taking into account Lefschetz sub-representations and by including factors
	of
	type III. The proof is inspired by \emph{loc. cit.}, but differs from it in
	that
	we focus on the representations $\mathds{S}_{\langle \lambda\rangle}V$\,: on
	the
	one hand, by Weyl's construction outlined in \S \ref{s:weyl}, we completely
	avoid resorting to understanding the irreducible sub-representations of tensor
	products as in \cite[Lemma 5.1.2]{hazama}\,; on the other hand, as explained
	before Lemma \ref{lem:isotropic}, these representations $\mathds{S}_{\langle
		\lambda\rangle}V$ are not in general the complexifications of sub-Hodge
	structures, as seems to be assumed in  \cite[Prop.~4.3]{hazama}. However, an important
	feature will be that these irreducible sub-representations are numerically
	Hodge
	symmetric.

	\begin{lem} \label{lem:hodgelenthrep}
		Let $V$ be an even-dimensional complex vector space with basis \linebreak
		$(e_1,\ldots, e_n, e_{-1},\ldots , e_{-n})$, and assume that $G$ is one of the
		following groups\,:
		\begin{enumerate}[(a)]
			\item $G= \operatorname{Sp}(V,\psi)$, where $\psi = \sum_{i=1}^n (e_i\otimes
			e_{-i} - e_{-i}\otimes e_i)$\,;
			\item $G= \operatorname{O}(V,\psi)$, where $\psi = \sum_{i=1}^n (e_i\otimes
			e_{-i} + e_{-i}\otimes e_i)$ and $n>1$.
		\end{enumerate}
		Consider the action of the torus $\operatorname{GL}_1$ on $V$ given by $z\cdot
		e_i = ze_i$ and $z\cdot e_{-i} = z^{-1}e_{-i}$ for $1\leq i \leq n$.
		Let $\lambda$ be a Young tableau of length $d$. Then $\mathds{S}_{\langle
			\lambda \rangle }V$ is numerically
		Hodge symmetric, that is, $\dim_\C (\mathds{S}_{\langle \lambda \rangle }V
		)^{p,q} = \dim_\C (\mathds{S}_{\langle \lambda \rangle }V )^{q,p}$ for all
		integers $p$ and $q$. 
		Moreover,	if $\mathds{S}_{\langle \lambda \rangle }V \neq 0$, then
		\begin{equation*}\label{eq:length}
			(\mathds{S}_{\langle \lambda \rangle }V )^{d,0} \neq 0.
		\end{equation*}
	\end{lem}
	\begin{proof}
		Our strategy of proof is taken from Hazama's proof of \cite[Prop.~4.3]{hazama}
		where the case $G=\operatorname{Sp}_{2n}$ was treated. Contrary to Hazama, we
		do not assume that $V$ is the complexification of a Hodge structure (since
		when extending scalars to
		$\C$ the irreducible representations of the Lefschetz group that arise are not
		Hodge
		structures). 
		
		We view $\mathds{S}_{\langle \lambda \rangle }V $ as a representation of the
		Lie algebra $\mathfrak{g}$. 	In both cases ($\mathfrak{g} =
		\mathfrak{sp}_{2n}$
		or $\mathfrak{so}_{2n}$), let us recall that, as in Bourbaki \cite[Chap. VIII,
		\S 13.3 \& \S 13.4]{bourbaki}, we let $E_{i,j}$ be the $2n\times 2n$ matrix
		expressed in the basis $(e_1,\ldots, e_n, e_{-1},\ldots , e_{-n})$ whose
		entries
		are all zero except for the $(i,j)$-th entry which is $1$. For $1\leq i \leq
		n$,
		the elements
		$$H_i :=  E_{i,i} - E_{-i,-i}$$ define a basis of a Cartan sub-algebra
		$\mathfrak{h}$ of $\mathfrak{g}$, and we let $(\varepsilon_1, \ldots,
		\varepsilon_n)$ be the dual basis of $(H_1,\ldots, H_n)$.
		
		Viewing  $\mathds{S}_{\langle \lambda \rangle }V$ as a subspace of $V^{\otimes
			d}$, we have the description 
		$$(\mathds{S}_{\langle \lambda \rangle }V)^{p,q} := \{w \in
		\mathds{S}_{\langle \lambda \rangle }V : H_0(w) = (p-q)w\}, \quad \mbox{where}
		\
		H_0 := \sum_i H_i.$$
		In particular, we have $\left( (\mathds{S}_{\langle \lambda \rangle }V
		)^{p,q} \right)^\vee =  (\mathds{S}_{\langle \lambda \rangle }V^\vee )^{q,p},$
		and
		since $V\simeq V^\vee$ as $\mathfrak{g}$-representations for our Lie algebras
		$\mathfrak{g}= \mathfrak{sp}_{2n}$ or $\mathfrak{so}_{2n}$, this immediately
		yields that 
		$\mathds{S}_{\langle \lambda \rangle }V$ is numerically Hodge symmetric.
		We also find that $\max \{p-q : (\mathds{S}_{\langle
			\lambda\rangle}V)^{p,q}\neq 0\}$ is equal to the maximum
		of the eigenvalues of $H_0$ acting on $\mathds{S}_{\langle \lambda \rangle
		}V$.
		We are going to show that if $\mathds{S}_{\langle \lambda \rangle }V \neq 0$,
		then $\max \{p-q : (\mathds{S}_{\langle \lambda\rangle}V)^{p,q}\neq 0\} = d$.

		First consider an irreducible representation $W$ of $\mathfrak{g} =
		\mathfrak{sp}_{2n}$ or $\mathfrak{so}_{2n}$ with highest weight~$\omega$. Let
		$v\in W$ denote one of its dominant vectors. Then for any element $H$ in the
		Cartan sub-algebra $\mathfrak{h}$ of $\mathfrak{g}$, we have 
		$$H(v) = \omega(H)v.$$ Denote $\alpha_i$ the simple roots of $\mathfrak{g}$.
		Specifically, if $\mathfrak{g} = \mathfrak{so}_{2n}$, then $\alpha_1 =
		\varepsilon_1 - \varepsilon_2, \ldots, \alpha_{n-1} = \varepsilon_{n-1} -
		\varepsilon_n, \alpha_n = 2\varepsilon_n$, and if $\mathfrak{g} =
		\mathfrak{sp}_{2n}$, then $\alpha_1 = \varepsilon_1 - \varepsilon_2, \ldots,
		\alpha_{n-1} = \varepsilon_{n-1} - \varepsilon_n, \alpha_n = \varepsilon_{n-1}
		+
		\varepsilon_n$.
		Since the weights of $W$ are of the form 
		$$\omega - \sum_{i=1}^n p_i\alpha_i$$
		for some nonnegative integers $p_i$, we find that 
		\begin{align*}
			\max \{p-q : (\mathds{S}_{\langle \lambda\rangle}V)^{p,q}\neq 0\} &= \max
			\left\{\left(\omega - \sum_i p_i\alpha_i\right)(H_0) : p_1,\ldots,p_n \geq 0
			\right\}\\
			&= \omega(H_0).
		\end{align*}
		
		In our case, by Proposition~\ref{prop:irred}, $\mathds{S}_{\langle \lambda
			\rangle }V$ corresponds either to an irreducible representation of
		$\mathfrak{g}$ with highest weight	$\lambda_1 \varepsilon_1 + \cdots +
		\lambda_n\varepsilon_n$, or in case $\mathfrak{g} = \mathfrak{so}_{2n}$ and
		$\lambda_n>0$ to the sum of two irreducible representations with highest
		weight
		$\lambda_1 \varepsilon_1 + \cdots + \lambda_n\varepsilon_n$ and $\lambda_1
		\varepsilon_1 + \cdots +
		\lambda_{n-1}\varepsilon_{n-1} - \lambda_n\varepsilon_n$.  In any case, we
		find
		that $	\max \{p-q : (\mathds{S}_{\langle \lambda\rangle}V)^{p,q}\neq 0\}$ is
		equal to $\omega(H_0)$, where $\omega = \lambda_1 \varepsilon_1 + \cdots +
		\lambda_n\varepsilon_n$, and hence is equal to $d:= \sum_i \lambda_i$.
	\end{proof}

	\begin{proof}[Proof of Theorem~\ref{thm:lefhodge}]
		We first note that it is enough to establish the theorem with complex
		coefficients. Precisely, we are going to show that for $H \subseteq
		\HH^k(A,\C)$ a
		$L(A)_\C$-sub-representation of Hodge level $\leq k-2n$, we have that $H$ is
		numerically Hodge symmetric and that 
		\begin{equation}\label{eq:strong}
			H\subseteq  \mathrm{Im} \left(\overline{\R}^n(A)_\C \otimes \HH^{k-2n}(A,\C)
			\stackrel{\cup}{\longrightarrow} \HH^k(A,\C)\right).
		\end{equation} 
		Here, by Hodge level we mean the following\,: since $L(A)_\C$ contains the
		circle group (defining the Hodge structure on $\HH_1(A,\Q)$), a
		$L(A)_\C$-sub-representation $H$ of $\HH^k(A,\C)$ has an eigenspace
		decomposition $\bigoplus_{p+q=k} H^{p,q}$, and the \emph{Hodge level} is then
		$\ell(H) := \max \{|p-q| : H^{p,q}\neq 0\}$.
		For ease of notation, we write from now on $\HH^*(-)_\C$
		for $\HH^*(-,\C)$.
		
		Second,  by considering the surjective homomorphism of Lefschetz
		representations \linebreak $\HH^1(A)^{\otimes k} \to \HH^k(A)$ given by
		cup-product, we
		note that we may assume that $H$ is an irreducible Lefschetz
		sub-representation
		of $\HH^1(A)^{\otimes k}_\C \subseteq \HH^k(A^k)_\C$. 
		
		Suppose then that  $A$ is isogenous to $A_1^{m_1}\times \cdots \times
		A_s^{m_s}$, where
		the $A_i$ are pairwise non-isogenous, simple abelian varieties. 
		The $L(A)_\C$-representation $\HH^1(A)_\C$ is isomorphic to the
		$(L(A_1)_\C\times \cdots \times L(A_s)_\C)$-representation
		$\HH^1(A_1)_\C^{\oplus m_1} \oplus \cdots \oplus \HH^1(A_s)_\C^{\oplus m_s}$,
		where each $L(A_i)_\C$ acts diagonally on $\HH^1(A_i)_\C^{\oplus m_i}$.
		If $H$ is an irreducible Lefschetz sub-representation of $\HH^1(A)^{\otimes
			k}_\C$, then up to permutation of the factors we may view $H$
		as an irreducible sub-representation of $\HH^1(A_1)_\C^{\otimes
			k_1}\otimes\cdots
		\otimes \HH_\C^1(A_s)^{\otimes k_s}$ for some non-negative integers $k_i$ such
		that
		$\sum k_i = k$. 
		Since $H$ is a $(L(A_1)_\C\times \cdots \times L(A_s)_\C)$-sub-representation,
		$H$
		must be of the form $$H = H_1\otimes \cdots\otimes H_s$$ for some irreducible
		$L(A_i)_\C$-sub-representations $H_i \subseteq \HH^1(A_i)_\C^{\otimes k_i}$. 
		
		With notations as in \S \ref{sec:albertlef}, $L(A_i)_\C$ is isomorphic to $f
		:=
		[F:\Q]$ copies of the group $G$, which is either the symplectic group (types I
		and II) or the orthogonal group (type III), and $\HH^1(A_i)_\C$ splits as the
		direct sum of $f$ copies of the standard representation $V$ of $G$ (type I) or
		as the direct sum of $2f$ copies of the standard representation $V$ of $G$
		(types II and III).
		Thus $H_i$ is an irreducible sub-representation of $V_1^{\otimes d_1} \otimes
		\cdots \otimes V_t^{\otimes d_t}$, where the $j$-th factor of $G^{\times t} =
		G\times \cdots \times G$ acts on $V_j$ as the standard representation and
		where
		the other factors act trivially. Hence, $H_i$ is of the form 
		$$H_i = H_{i,1}\otimes\cdots \otimes H_{i,t}$$
		for some irreducible $G$-sub-representations  $H_{i,j} \subseteq V_j^{\otimes
			d_j}$.
		
		Now, by Proposition~\ref{prop:irred}, each $H_{i,j}$ must be of the form 
		\begin{equation*}\label{eq:Hij}
			H_{i,j} = \Psi_{I_1} \circ \cdots \circ\Psi_{I_{k_{i,j}}}
			(\mathds{S}_{\langle \lambda_{i,j}\rangle}V_j)
		\end{equation*}
		for some Young tableau  $\lambda_{i,j}$ and some pairs of integers
		$I_1,\ldots,
		I_{k_{i,j}}$.
		From Lemma~\ref{lem:isotropic}, we know that $V_j$ decomposes as $V_j^{1,0}
		\oplus V_j^{0,1}$ in such a way that both $V_j^{1,0}$ and $V_j^{0,1}$ are
		isotropic for the non-degenerate bilinear form on $V_j$ (which is
		skew-symmetric
		for types I and II, and symmetric for type III).
		The assumptions of Lemma~\ref{lem:hodgelenthrep} are thus met for $V_j$, and
		we therefore see that $H_{i,j}$ is numerically Hodge symmetric and satisfies
		$(H_{i,j})^{d_{i,j},0}\neq 0$, where $d_{i,j}$ is the
		length of the Young tableau $\lambda_{i,j}$. We deduce that $H$ is numerically
		Hodge symmetric and satisfies
		$$\ell(H) = \sum_{i,j} d_{i,j}.$$
		Now we can conclude, because composing with $\Psi_I$ amounts to cupping with a
		divisor with complex coefficients.
	\end{proof}

	\begin{rmk}\label{rmk:limit}
		Our method for establishing the generalized Hodge conjecture for Lefschetz
		sub-representations  of abelian varieties of totally real type, which in fact
		consists in establishing it after extending the scalars to $\C$, is too crude to
		work for powers of simple abelian varieties of type~IV. 
		Let us briefly describe a
		simple example. Beforehand, on a positive note, we simply mention that the method
		works for powers of a CM elliptic curve. Let then $A$ be a simple abelian
		surface of type~IV\,; it is known that $A$ must be of CM type, so that its
		Lefschetz group  is $\operatorname{GL}(1)\times \operatorname{GL}(1)$  after
		extending the scalars to $\C$ and we may write $\HH^1(A,\C) = (V\oplus
		V^\vee)\oplus (W\oplus W^\vee)$, where $V= V^{1,0}$ and $W=W^{1,0}$ and
		$\operatorname{GL}(1)\times \operatorname{GL}(1)$  acts on $V$ via the first
		projection and on $W$ via the second projection. Consider then for instance the
		1-dimensional $L(A)_\C$-sub-representation $V\otimes W^\vee$ inside $\HH^2(A,\Q)
		= \bigwedge^2\HH^1(A,\Q)$. On the one hand, it has Hodge type $(1,1)$ but is not
		acted upon trivially by $L(A)_\C$ and thus is not spanned by a Hodge class. That
		type of phenomenon does not occur for abelian varieties of totally real type
		because their Lefschetz representations are numerically Hodge symmetric (Lemma
		\ref{lem:isotropic}). On the other hand, we deduce that the Galois orbit of
		$V\otimes W^\vee$ inside $\HH^2(A,\Q)$ has Hodge length~$2$\,; this suggests
		that it is not straightforward to read the Hodge length of the Galois
		closure of a $L(A)_\C$-sub-representation from its Hodge length without
		resorting to a detailed Galois analysis.
	\end{rmk}

	\subsection{Lefschetz representations and the generalized Hodge conjecture
		II}\label{s:GHC2}
	In this section, we would like to improve slightly on Theorem
	\ref{thm:lefhodge} by allowing our abelian varieties to be isogenous to the
	product of an abelian variety of totally real type with some power of an
	abelian
	surface of CM type, or with the product of powers of three elliptic curves. Our
	main result is Proposition~\ref{prop:projlef}. In particular, we recall a
	strong
	version of the generalized Hodge conjecture for self-powers of abelian
	surfaces\,; see Corollary~\ref{cor:GHCsurface}.
	\medskip
	
	Let us start with the case where our abelian varieties have no factor of
	totally real type.
	The following theorem is due to Abdulali \cite[Examples 2 \& 3]{abdulali}\,:
	
	\begin{thm}[Abdulali \cite{abdulali}, strong GHC for powers of CM abelian
		surfaces and certain products of CM elliptic curves]
		\label{thm:GHC}
		
		Let $A$ be an abelian variety that is isogenous to either
		\begin{enumerate}[(i)]
			\item $E_1^{k_1} \times E_2^{k_2} \times E_3^{k_3}$ for some CM elliptic
			curves $E_i$, or 
			\item the power of a CM abelian surface $S$.
		\end{enumerate} 
		Let $H \subseteq
		\HH^k(A,\Q)$ be a
		Hodge sub-structure of Hodge level $\leq k-2n$. Then 
		$$ H\subseteq  \sum_B \mathrm{Im} \left({\R}^{\dim B +n}(A\times B) \otimes
		\HH^{k-2n}(B,\Q) \longrightarrow \HH^k(A,\Q)\right),$$
		where $\Gamma\otimes \gamma \mapsto \Gamma^*(\gamma)$ and where the sum runs
		over all abelian varieties $B$.
	\end{thm}
	\begin{proof} For a proof, we refer to Abdulali \cite{abdulali}. Let us mention
		that in case $(i)$ the sum can be taken over abelian varieties of the form
		$E_1^{m_1} \times E_2^{m_2} \times E_3^{m_3}$, and in case $(ii)$ over powers
		of
		$S$, unless
		$S$ is an abelian surface
		with CM by a field $E$ not Galois over $\Q$, in which case, denoting $S'$ the
		other abelian surface with CM by $E$, the sum runs through abelian varieties
		of the form $S^i \times (S')^j$.  That the correspondences in the sum can be
		chosen to be in $\R^*$ is due to the fact that for abelian varieties of the
		form
		$E_1^{m_1} \times E_2^{m_2} \times E_3^{m_3}$, or $S^i \times (S')^j$ as
		above,
		the Hodge group coincides with the Lefschetz group, so that all Hodge classes
		on
		$E_1^{m_1} \times E_2^{m_2} \times E_3^{m_3}$ or $S^i \times (S')^j$ belong in
		fact to $\overline{\R}^*$.
	\end{proof}

	\begin{prop}\label{prop:R} We have the following three statements\,:
		\begin{enumerate}[(a)]
			\item Let $A$ be an abelian variety and let $m$ be a positive integer. Then
			any
			symmetrically distinguished cycle on $A^m$ that is generically defined for
			$m$-fold powers of polarized abelian varieties belongs to $\R^*(A^m)$. 
			
			\item
			Let $A$, $B$ and $C$ be abelian varieties, and let $\gamma \in \R^*(A\times
			B)$
			and $\gamma'\in \R^*(B\times C)$ be two correspondences. Then $\gamma'\circ
			\gamma$ belongs to $\R^*(A\times C)$.
			\item Let $f: A\to B$ be a homomorphism of abelian varieties. Then the graph
			$\Gamma_f$ of $f$ belongs to $\R^*(A\times B)$.
		\end{enumerate}
	\end{prop}
	\begin{proof}
		By O'Sullivan's Theorem~\ref{T:Osullivan}, a symmetrically distinguished cycle
		in $\DCH^*(A)$ whose cohomology class belongs to $\overline{\R}^*(A)$ belongs
		to
		$\R^*(A)$ for any abelian variety $A$. Thus (a) follows from the fact that the cohomology class of a
		generically defined cycle on the $m$-fold power of a polarized abelian variety $A$ belongs to
		$\overline{R}^*(A^m)$\,;  see Theorem
		\ref{T:gen2}. (More precisely,  Hodge classes on $A^m$, with $A$ a very
		general
		polarized abelian variety, consist of polynomials in $p_i^*L$ and
		$p_{i,j}^*c_1(\mathcal{P_A})$, where $L$ is the polarization of $A$.)
		
		For (b), observe that the composition of two correspondences in
		$\overline{\R}^*$ yields a correspondence in $\overline{\R}^*$\,; indeed a
		correspondence belongs to $\overline{\R}^*$ if and only if it commutes with
		the
		action of the Lefschetz group.
		Case (b) then follows from this fact together with  the fact that by
		O'Sullivan's Theorem $\gamma'\circ \gamma$ is symmetrically distinguished
		(since
		$\gamma$ and $\gamma'$ are).
		
		For (c), since $\Gamma_f = (f,\mathrm{id}_B)^*\Delta_B$, it suffices to show
		that $\Delta_B \in \R^*(B\times B)$. This can be found in \cite[\S 5]{scholl}.
	\end{proof}

	As a consequence of Theorems~\ref{thm:lefhodge} and \ref{thm:GHC}, we have the
	following analogue of Proposition~\ref{P:niveau}, which in particular establishes Conjecture~\ref{conj:strongGHC} for Lefschetz sub-representations of certain abelian varieties\,:
	
	\begin{prop}\label{prop:projlef}
		Let $A$ be a complex abelian variety of dimension $g$, and let $H \subseteq
		\HH^k(A,\Q)$ be a
		Lefschetz sub-representation of Hodge level $\leq k-2n$. Assume that $A$ is
		isogenous to $A_0 \times A_1$ with
		\begin{itemize}
			\item $A_0$ isomorphic to $E_1^{k_1} \times E_2^{k_2} \times E_3^{k_3}$ for
			some
			CM elliptic curves $E_i$, or to the power of a CM abelian surface\,; 
			\item $A_1$ isomorphic to an abelian variety of totally real type (\emph{cf.}
			Definition~\ref{def:totreal}).
		\end{itemize} 
		Then there exists an idempotent correspondence $p_H \in \R^g(A\times A)$
		inducing the  projection
		$\HH^*(A,\Q) \to H \to \HH^*(A,\Q)$, which is a
		linear combination of correspondences of the form $$\mathfrak{h}(A)
		\stackrel{\rho}{\longrightarrow} \mathfrak{h}(B)(n) 
		\stackrel{\zeta}{\longrightarrow} \mathfrak{h}(A),$$ for some abelian
		varieties $B$ and some correspondences $\rho$ and
		$\zeta$ that belong to $\R^*(A\times B)$ and $\R^*(B\times A)$, respectively. 
		
	\end{prop}
	\begin{proof} First we show a strong version of the generalized Hodge
		conjecture for Lefschetz sub-representations of $A$\,; namely, we show that
		\begin{equation}\label{eq:prel}
			H\subseteq  \sum_B \mathrm{Im} \left({\R}^{\dim B +n}(A\times B) \otimes
			\HH^{k-2n}(B,\Q) \longrightarrow \HH^k(A,\Q)\right),
		\end{equation}
		where $\Gamma\otimes \gamma \mapsto \Gamma^*(\gamma)$ and where the sum runs
		over all abelian varieties $B$. As outlined after the proof of 
		\cite[Prop.~4]{abdulali2} in the context of Hodge sub-structures, there is a
		slight subtlety\,: one needs to use the stronger statement of Theorem
		\ref{thm:lefhodge} described in its proof, namely, that for $H_1 \subseteq
		\HH^{k_1}(A_1,\C)$ a
		$L(A_1)_\C$-sub-representation of level\footnote{See the proof of
			Theorem~\ref{thm:lefhodge} for the notion of \emph{level} of a
			$L(A_1)_\C$-sub-representation.} $\leq k_1-2n_1$,
		we have that $H_1$ is
		numerically Hodge symmetric and that 
		\begin{equation}\label{eq:H1}
			H_1\subseteq  \mathrm{Im} \left(\overline{\R}^{n_1}(A_1)_\C \otimes
			\HH^{k_1-2n_1}(A_1,\C)
			\stackrel{\cup}{\longrightarrow} \HH^{k_1}(A_1,\C)\right).
		\end{equation} 
		
		Let $H$ be an irreducible  Lefschetz sub-representation of 	$\HH^k(A_0\times
		A_1,\Q)$. Then $H$ is a sub-representation of $L(A_0\times A_1) \cong
		L(A_0)\times L(A_1)$ acting on some K\"unneth component $\HH^{k_0}(A_0,\Q)
		\otimes
		\HH^{k_1}(A_1,\Q)$ for some $k_0+k_1=k$. Let then $V$ be an irreducible
		sub-representation of $L(A_0\times A_1)_\C$ acting on $H_\C$. It is of the
		form
		$V_0 \otimes_\C V_1$ for some $L(A_0)_\C$-sub-representation $V_0 \subseteq
		\HH^{k_0}(A_0,\C)$ and  some $L(A_1)_\C$-sub-representation $V_1 \subseteq
		\HH^{k_1}(A_1,\C)$. Moreover  the Galois conjugates of
		$V_0\otimes_\C V_1$ span $H_\C$\,; indeed, the span is defined over $\Q$ and
		defines a non-trivial sub-representation of the irreducible $L(A_0\times
		A_1)$-representation $H$.  The subspace spanned by the Galois conjugates of
		$V_0$ inside
		$\HH^{k_0}(A_0,\C)$ is defined over $\Q$\,; we denote it $W_0$. Then $W_0$ is
		a
		$L(A_0)$-sub-representation of $\HH^{k_0}(A_0,\Q)$. 
		
		We note from Theorem~\ref{thm:lefhodge} and its proof that $V_1$ and its Galois
		conjugates $V_1^\sigma$ are Hodge symmetric of same level. We find
		\begin{align*}
			\ell(H) &= \max_\sigma \ell(V_0^\sigma \otimes V_1^\sigma)\\
			&= \max_\sigma \left( \ell(V_0^\sigma) + \ell(V_1^\sigma)\right)\\
			&= \max_\sigma \ell(V_0^\sigma) + \ell(V_1)\\
			&= \ell(W_0) + \ell(V_1).
		\end{align*}
		Here the maximum is taken over all elements $\sigma \in
		\operatorname{Aut}_\Q(\C)$, and the second equality holds because $V_1^\sigma$
		is Hodge symmetric.
		Let us then write $$\ell(W_0) = k_0 - 2n_0\quad \mbox{and}\quad \ell(V_1) = k_1
		- 2n_1, \quad \mbox{for } n_1+n_2=n.$$
		
		By the above \eqref{eq:H1}, there are an integer $s$ and correspondences
		$\Gamma_{r,1} \in
		\R^*(A_1 \times A_1)_\C$, $1\leq r\leq s$, such that 
		$$V_1 \subseteq \sum_r \Gamma_{r,1,*} \HH^{k_1-2n_1}(A_1,\C).$$
		Since each $\Gamma_{r,1}$ is a $\C$-linear combination of elements in
		$\R^*(A_1\times A_1)$, up to increasing $s$, we may assume that each
		$\Gamma_{r,1}$ is in fact in the image of $\R^*(A_1\times A_1)
		\hookrightarrow \R^*(A_1 \times A_1)_\C$, so that for every $\sigma \in
		\mathrm{Aut}(\C)$, we have 
		$$V_1^\sigma \subseteq \sum_r \Gamma_{r,1,*} \HH^{k_1-2n_1}(A_1,\C).$$
		On the other hand, there are finitely many non-zero correspondences
		$\Gamma_{B,0} \in \R^*(B\times A_0)$ indexed by abelian varieties $B$, such
		that
		$$W_0 \subseteq \sum_B \Gamma_{B,0,*} \HH^{k_0-2n_0}(B,\Q).$$
		Since $V_0\subseteq W_{0,\C}$,	we have 
		\begin{align*}
			H_\C &= \sum_{\sigma \in \mathrm{Aut}(\C)} (V_0\otimes_\C V_1)^\sigma
			\subseteq \sum_{\sigma,\tau \in \mathrm{Aut}(\C)} (V_0^\sigma \otimes_\C
			V_1^\tau) = W_{0,\C}\otimes\sum_{\tau \in \mathrm{Aut}(\C)} V_1^\tau \\
			&\subseteq  \big( \sum_B \Gamma_{B,0,\C,*}\HH^{k_0-2n_0}(B,\C)\big) \otimes
			\big(\sum_{ r} \Gamma_{r,1,\C,*}\HH^{k_1-2n_1}(A_1,\C)\big) \\
			&\subseteq \sum_{B,r} \big(\Gamma_{B,0}\otimes
			\Gamma_{r,1}\big)_{\C,*}\HH^{k_0+k_1-2(n_0+n_1)}(B\times A_1,\C).
		\end{align*}	
		This establishes \eqref{eq:prel}.
		
		Now, since $H$ is a Lefschetz sub-representation of $\HH^k(A,\Q)$,
		there exists by Lemma~\ref{lem:projlef} an idempotent  $\pi_H \in \R^g(A\times
		A)$ such that $(\pi_H)_*\HH^k(A,\Q) = H$. Composing $\pi_H$  with the
		correspondence $\sum \Gamma_{B,0}\otimes
		\Gamma_{r,1}$, we see that 
		$$H = \Big(\pi_H \circ  \sum_{B,r} \Gamma_{B,0}\otimes
		\Gamma_{r,1}\Big)_* \HH^{k-2n}(B\times A_1,\Q).$$
		In order to conclude the proof of Proposition~\ref{prop:projlef}, we observe
		that we may proceed as in the  proof of Proposition~\ref{P:niveau}\,; indeed,
		all the correspondences appearing there are compositions of correspondences in
		$\R^*$, and therefore thanks to Proposition~\ref{prop:R} belong to
		$\R^*$. 
	\end{proof}

	As a corollary,
	let us mention the
	following result, \emph{cf.} \cite[8.1(2)]{abdulali3}.
	
	\begin{cor}[strong GHC for self-powers of elliptic curves, or abelian
		surfaces]\label{cor:GHCsurface}
		Let $A$ be an abelian variety of dimension $\leq 2$, and let $m$ be a positive
		integer. 
		Let $H \subseteq
		\HH^k(A^m,\Q)$ be a
		Hodge sub-structure of Hodge level $\leq k-2n$. Then 
		$$ H\subseteq  \sum_B \mathrm{Im} \left({\R}^{\dim B +n}(A^m\times B) \otimes
		\HH^{k-2n}(B,\Q) \longrightarrow \HH^k(A^m,\Q)\right),$$
		where $\Gamma\otimes \gamma \mapsto \Gamma^*(\gamma)$ and where the sum runs
		over all abelian varieties $B$.
	\end{cor}
	\begin{proof}
		The case where $A$ has CM was covered in Abdulali's Theorem~\ref{thm:GHC}, while the case
		where $A$ is without CM is covered by Theorem~\ref{thm:lefhodge} (recall that
		in
		these cases, $\mathrm{Hdg}(A) = L(A)$). Thus it only remains to treat the case
		where $A = E\times E'$, where $E$ is an elliptic curve without CM and $E'$ is
		an
		elliptic curve with CM. In that case, we still have $\mathrm{Hdg}(A) = L(A)$
		(see e.g. \cite{MZ}) and one concludes with Proposition~\ref{prop:projlef}.
	\end{proof}

	\subsection{Lefschetz representations and the generalized Bloch conjecture} We
	are now in a position to prove the theorem announced in \S \ref{sec:lefintro} of  the introduction.
	
	\begin{thm}\label{T:mainLef}
		Let $A$ and $A'$ be two abelian varieties, and let $\gamma$ be a cycle in
		$\mathrm{R}^*(A\times A')$.   Assume that $A$ is isogenous to $A_0 \times A_1$
		with
		\begin{itemize}
			\item $A_0$ isomorphic to $E_1^{k_1} \times E_2^{k_2} \times E_3^{k_3}$ for
			some CM elliptic curves $E_i$, or to the power of a CM abelian surface\,; 
			\item $A_1$ isomorphic to an abelian variety of totally real type (\emph{cf.}
			Definition~\ref{def:totreal}).
		\end{itemize} 
		If
		$\gamma^*\HH^{i,j}(A')=0$ for all $j<n$, then
		$\gamma_*\CH_r(A) = 0$ for all $r<n$.
	\end{thm}
	\begin{proof}
		Since $\gamma$ is a cycle in
		$\mathrm{R}^*(A\times A')$, we have that $\gamma^*\HH^*(A',\Q)$ is a Lefschetz
		sub-representation $H$ of $\HH^*(A,\Q)$. By the assumption
		$\gamma^*\HH^{i,j}(A')=0$ for all $j<n$ and by Proposition~\ref{prop:projlef},
		we
		see that, modulo homological equivalence, $\gamma =  \gamma \circ p_H$ is a
		linear combination of cycles in $\R^*(A\times A')$ that factor as
		$$\mathfrak{h}(A)
		\stackrel{\rho}{\longrightarrow} \mathfrak{h}(B)(n) 
		\stackrel{\zeta}{\longrightarrow} \mathfrak{h}(A'),$$ 
		for some abelian
		varieties $B$ and some correspondences $\rho$ and
		$\zeta$ that belong to $\R^*(A\times B)$ and $\R^*(B\times A')$, respectively.
		Since all the correspondences involved belong to $\R^*(-)$, O'Sullivan's
		Theorem
		\ref{T:Osullivan} tells us that the latter in fact holds modulo rational
		equivalence. It follows that $\gamma$ factors through a morphism 
		$\mathfrak{h}(A) \to \bigoplus_B \mathfrak{h}(B)(n)$, where the direct sum
		runs
		through the abelian varieties that appeared above. In particular, the map
		$\gamma_* : \CH_r(A) \to \CH_{*}(A')$ factors through a map $\CH_r(A) \to
		\bigoplus_B \CH_{r-n}(B)$, and hence $\gamma_* : \CH_r(A) \to \CH_{*}(A')$ is
		zero for $r<n$.
	\end{proof}

	\begin{rmk}\label{R:mainLef}
		In the case where $A$ is isogenous to the power of an abelian variety of
		dimension $\leq 2$, we will use Corollary \ref{cor:GHCsurface} to prove in Theorem
		\ref{T:app} that
		if  $\gamma$ is a cycle in
		$\mathrm{CH}^*(A\times A)$ such that
		$\gamma^*\HH^{i,j}(A)=0$ for all $j<n$, then
		$\gamma_*$ acts nilpotently on $\CH_r(A)$ for all $r<n$.
	\end{rmk}

	\section{Applications}\label{S:ex}
	
	The simplest form of Bloch's conjecture predicts that if a smooth projective
	complex variety $X$ satisfies $h^{i,0}(X) = 0$ for all positive integers $i$,
	then $\CH_0(X) = \Q$. If now $S$ is a smooth projective complex surface that
	satisfies  $h^{1,0}(S) = 0$  and $h^{2,0}(S) = 1$, then since
	$\bigwedge^2h^{2,0}(S) = 0$ one would expect that $a\times b = b\times a$ in $\CH_0(S\times S)$ for all zero-cycles $a, b\in \CH_0(S)_{\mathrm{num}}$, where $\CH_0(S)_{\mathrm{num}}$
	denotes
	the zero-cycles of degree zero. This expectation was studied by Voisin in
	\cite{voisin0} who conjectured it for K3 surfaces, and established it for
	Kummer
	surfaces and for a certain 10-dimensional family of K3 surfaces\,; see also
	\cite{laterveer1, laterveer2}. Another prediction of Bloch's conjecture is the
	following. Let $f:X\to X$ be an automorphism of a smooth projective variety such
	that $f^*$ acts as the identity on $\HH^0(\Omega_X^i)$ for all~$i$\,; then $f$
	should act unipotently on $\CH_0(X)$. This was checked for finite-order
	automorphisms of K3 surfaces by Voisin \cite{voisink3} and Huybrechts
	\cite{huybrechts}.
	
	In this section, we  answer questions of that
	type for curves, abelian varieties, Kummer surfaces and generalized Kummer varieties.
	In \S \S \ref{sec:0ab}, \ref{s:lin} and \ref{s:voisin}, we use our results on
	generically defined cycles, while in \S \S \ref{sec:motab} and
	\ref{sec:symplecto}, we use the strong form of the generalized Hodge conjecture
	for powers of abelian surfaces.

	\subsection{Symmetric and skew-symmetric cycles  on powers of curves or of abelian varieties}\label{sec:0ab}
	Recall from Shermenev \cite{shermenev} and Deninger--Murre \cite{dm} that the Chow motive of an abelian variety $A$ of dimension $g$ admits a weight decomposition 
	$$\mathfrak{h}(A) = \bigoplus_{i=0}^{2g} \mathfrak{h}^i(A)$$ with the property that 
	\begin{equation}\label{eq:eigen}
\CH^j(\mathfrak{h}^{i}(A)) = \{a\in \CH^i(A) : [n]^*a = n^{i}a \ \text{for all} \ n\in \Z\},
	\end{equation}
	 where $[n]:A\to A$ is the multiplication-by-$n$ homomorphism, and the property that the diagonal embedding $A\hookrightarrow A^i$ induces a canonical isomorphism
	$$\mathfrak{h}^i(A) \cong \mathrm{S}^i \mathfrak{h}^1(A),$$
where the right-hand term denotes the $i$-th symmetric power of the motive $\mathfrak{h}^1(A)$, seen as a direct summand of the motive of $A^i$.
\medskip

	The following result generalizes to integers $i\neq g$ a result of Voisin
	\cite[Example 4.40]{voisinbook}. Note that in the proof of \emph{loc. cit.},
	one
	has to check that $\sigma$ sends $\mathfrak{h}^g(A)\otimes
	\mathfrak{h}^g(A)$ into $\mathfrak{h}^g(A)\otimes \mathfrak{h}^g(A)$ (\emph{a priori}
	$\sigma$ sends $\mathfrak{h}^g(A)\otimes \mathfrak{h}^g(A)$ into
	$\mathfrak{h}^{2g}(A\times A) = \bigoplus_{i}\mathfrak{h}^i(A)\otimes
	\mathfrak{h}^{2g-i}(A)$).
	\begin{thm}\label{T:sym}
		Let $A$ be an abelian variety of dimension $g$. Let $i$ be a nonnegative
		integer.
		\begin{itemize}
			\item For $i$ odd, we have $\CH_0\left( \mathrm{S}^{N}\mathfrak{h}^{2g-i}(A)
			\right) = 0$ for all $N > \binom{g}{i}$.
			\item For $i$ even, we have $\CH_0\left( \bigwedge^{N}\mathfrak{h}^{2g-i}(A)
			\right) = 0$ for all $N > \binom{g}{i}$.
		\end{itemize}
		In particular, if $N > \binom{g}{i}$ and if $a_j$, $1\leq j \leq N$, are
		zero-cycles on $A$ such that $[n]_*a_j = n^i a_j$ for all integers $n$, then the following holds.
		\begin{itemize}
			\item For $i$ odd, we have $\sum_{\sigma \in \mathfrak{S}_N} a_{\sigma(1)}
			\times \cdots\times a_{\sigma(N)} = 0$  in $\CH_0(A^N)$.
			\item For $i$ even, we have $\sum_{\sigma \in \mathfrak{S}_N}
			\mathrm{sgn}(\sigma)\,  a_{\sigma(1)} \times \cdots\times a_{\sigma(N)} = 0$ 
			in
			$\CH_0(A^N)$.
		\end{itemize}
	\end{thm}
	\begin{proof} The reason for considering symmetric or anti-symmetric powers
		when $i$ is odd or even, respectively, is because the cohomology ring of a
		smooth variety is graded-commutative. As for 	
		the second part of the theorem, this follows simply from the description \eqref{eq:eigen} of 
		$\CH_0(\mathfrak{h}^{2g-i}(A))$.

		Given a permutation $\sigma \in \mathfrak{S}_{N}$, let us denote
		$\Gamma_\sigma
		\in \CH^{Ng}(A^N \times A^N)$ the graph of the morphism $(x_1,\ldots,x_n) \to
		(x_{\sigma^{-1}(1)},\ldots , x_{\sigma^{-1}(n)})$. The symmetric projector and
		the alternate projector are respectively
		\begin{equation}\label{E:proj}
			p_{S^N} := \frac{1}{n!}\sum_{\sigma \in \mathfrak{S}_N} \Gamma_\sigma \quad
			\text{and} \quad p_{\wedge^N} := \frac{1}{n!}\sum_{\sigma \in \mathfrak{S}_N}
			\mathrm{sgn}(\sigma)\, \Gamma_\sigma\,;
		\end{equation}
		they are generically defined idempotent correspondences for $N$-fold products
		of polarized
		abelian varieties of dimension $g$.
		For $i$ odd, the generically defined correspondence $p_{S^N}\circ
		(\pi^{2g-i}_{A} \otimes \cdots \otimes \pi^{2g-i}_{A})$ acts trivially on
		$\HH^{N(2g-i),0}(A^N)$ for $N> \binom{g}{i}$. For $i$ even, the generically
		defined correspondence $p_{\wedge^N}\circ (\pi^{2g-i}_{A} \otimes \cdots
		\otimes \pi^{2g-i}_{A})$ acts trivially on $\HH^{N(2g-i),0}(A^N)$ for $N>
		\binom{g}{i}$. In both case, we conclude by invoking Theorem~\ref{T:main}.
	\end{proof}

	\begin{rmk}
	Of course, one can state and prove many variants of Theorem \ref{T:sym}. For example, given integers $n\leq i$ with say $i$ odd, since the Hodge numbers $$h^{{N(2g-i),0}}, h^{{N(2g-i)-1,1}},\ldots, h^{{N(2g-i)-n,n}}$$ of $\mathrm{S}^{N}\mathfrak{h}^{2g-i}(A)$ vanish for $N> \sum_{j=0}^n \binom{g}{j} \binom{g}{i-j}$, we can prove that $\CH_r(\mathrm{S}^{N}\mathfrak{h}^{2g-i}(A))=0$ for all $r\leq n$. One could also consider the motives $\bigwedge^MS^N\mathfrak{h}^{2g-i}(A)$, various images under Schur functors, \emph{etc.}
Via the Abel--Jacobi map, one also recovers the fact that 
for a smooth projective  curve $C$ of genus $g$ we have 
$$ \sum_{\sigma\in\mathfrak{S}_N} a_{\sigma(1)} \times\cdots \times
a_{\sigma(N)} = 0 \quad \mbox{in } \CH_0(C^N),$$ for any integer
$N>g$ and any degree-0 zero-cycles $a_1,\ldots, a_N \in \CH_0(C)$. This is originally due independently to Voisin~\cite[p.267]{voisin0} and Voevodsky~\cite{voevodsky}\,; since algebraically trivial cycles are parametrized by curves, this establishes that, for any smooth projective variety $X$, any algebraically trivial cycle $a \in \CH^r(X)$ is \emph{smash-nilpotent}, that is, $a\times \cdots\times a = 0 \in \CH^{rN}(X^N)$ for some $N>0$.
	\end{rmk}

	\subsection{Zero-cycles on generalized Kummer varieties}\label{s:lin}
	
	Let $A$ be an abelian surface. The $n$-th generalized Kummer variety $K_n(A)$ associated
	to $A$ is a fiber of the isotrivial fibration $\mathrm{Hilb}^{n+1}(A) \to A$
	that is the composite of the Hilbert--Chow morphism $\mathrm{Hilb}^{n+1}(A) \to
	A^{n+1}/\mathfrak{S}_{n+1}$ with the sum morphism $\Sigma :
	A^{n+1}/\mathfrak{S}_{n+1} \to A$. 
	The variety $K_n(A)$ is known to be a hyperK\"ahler variety \cite{beauvillec1}, in particular $h^{2i,0}(K_n(A)) = 1$ for $0\leq i \leq n$, and $h^{2i+1,0}(K_n(A)) = 0$ for all $i$.
	A generalized Kummer
	variety of dimension 2 is nothing but a Kummer surface.

	In \cite{ftv}, we established that the Chow ring $\CH^*(K_n(A))$ of generalized
	Kummer varieties admits a grading that splits the conjectural Bloch--Beilinson
	filtration. We write $$\CH^*(K_n(A)) = \bigoplus_j \CH^*(K_n(A))_{(j)}.$$ In the case
	of zero-cycles, this grading has the following simple description (see \cite{lin}).
	The restriction of the Hilbert--Chow morphism provides
	a
	birational morphism from  $K_n(A)$ to the variety $A_0^{n+1} /
	\mathfrak{S}_{n+1}$, where $A_0^{n+1}$ is the fiber over $0$ of the sum
	morphism
	$\Sigma : A^{n+1} \to A$ and the action of the symmetric group
	$\mathfrak{S}_{n+1}$ is the one induced from the action on $A^{n+1}$ permuting
	the factors. Then
	$\CH_0(K_n(A))_{(j)}$ identifies with 
	$(\CH_0(A_0^{n+1})_{(j)})^{\mathfrak{S}_{n+1}}$ via the restriction of the
	Hilbert--Chow morphism, where $\CH_0(A_0^{n+1})_{(j)}$ is defined in
	\eqref{E:beauville}. Let us identify $A_0^{n+1}$ with $A^n$, and let us
	write
	$$p:= \frac{1}{(n+1)!}\sum_{\sigma \in \mathfrak{S}_{n+1}} \Gamma_\sigma $$
	for the projector on the $\mathfrak{S}_{n+1}$-invariant part of the motive of
	$A^n$\,; it is a generically defined correspondence for $n$-fold products of
	polarized abelian surfaces.
	Then we have $$\CH_0(K_n(A))_{(j)} = p^*\CH_0(A^n)_{(j)} = (\pi_{A^n}^{j} \circ p)^*
	\CH_0(A^n),$$
	where $\pi_{A^n}^{j}$ is a Chow--K\"unneth projector as in Lemma
	\ref{L:kleiman}, in particular generically defined. \medskip

	The following theorem is due to Hsueh-Yung Lin. We provide a short proof based on our Theorem~\ref{T:main}.
	
	\begin{thm}[Lin \cite{lin, lin2}] \label{T:Lin}
		$\CH_0(K_n(A))_{(2j+1)} = 0$ for all integers $j$.
	\end{thm}
	\begin{proof}
		We know that $ p^*\HH^{2j+1,0}(A^n) = \HH^{2j+1,0}(K_n(A)) = 0$ for all
		integers
		$j$, so that $ (\pi^{2j+1}_{A^n} \circ p)^*\HH^*(A^n,\Q) \subseteq
		\mathrm{N}_H^1 \HH^{*}(A^n,\Q)$.  The theorem is then 
		a straightforward application of Theorem~\ref{T:main}.
	\end{proof}
	
	The following theorem generalizes a result of Voisin \cite[Proposition
	3.2]{voisin0} for Kummer surfaces to the higher dimensional case of generalized
	Kummer varieties.
	
	\begin{thm} \label{T:Alt2Kum}
		Let $a$ and $b$ be two cycles in $\CH_0(K_n(A))_{(2j)}$. Then 
		$$a\times b  = b\times a \quad \text{in}\ \CH_0(K_n(A)\times K_n(A)).$$
	\end{thm}
	\begin{proof}
		Let
		$p_{\wedge^2}$ be the generically defined idempotent $\Delta_{A^{n}} -
		\Gamma_\tau \in \CH^{2n}(A^{n}\times A^{n})$ where $\tau  : A^n\times A^n \to
		A^{n} \times A^{n}$ is the morphism permuting the factors.  Since
		$\HH^{2j,0}(A^n)^{\mathfrak{S}_{n+1}}  = \HH^{2j,0}(K_n(A))  = 1$, we have
		$(p_{\wedge^2} \circ \pi^{2j}_{A^n} \circ p)^*\HH^*(A^n,\Q) \subseteq
		\mathrm{N}_H^1 \HH^{*}(A^n,\Q)$. We may now conclude by invoking Theorem
		\ref{T:main}.
	\end{proof}

	\subsection{On a conjecture of Voisin} \label{s:voisin}
	Let $N\geq 2$ be an integer, and let $S$
	be a K3 surface. Denote $pr : S^N \to S^{N-1}$ the projection to the first
	$N-1$
	factors\,; it induces for all $l\geq 0$ a morphism 
	$$pr_* : (p_{\wedge^N})_* \CH_l(S^N) \longrightarrow (p_{\wedge^{N-1}})_*
	\CH_l(S^{N-1}),$$
	where $p_{\wedge^{N}}$ is the anti-symmetrization projector defined in
	\eqref{E:proj}.
	As a consequence of the Bloch--Beilinson philosophy, Voisin \cite[Conjecture
	3.9]{voisin0} stated\,:
	
	\begin{conj}[Voisin] \label{C:voisin} The anti-symmetrization projector
		$p_{\wedge^{N+1}}$ acts as zero on 
		$\ker (pr_*) \otimes \CH_0(S)_{\mathrm{num}}$ for all $l <  N$. 
	\end{conj}
	Voisin established this conjecture for $N=2$ in the case where $S$ is a Kummer
	surface by a lengthy calculation\,; see \cite[Theorem 3.10]{voisin0}. A variant
	of our Theorem~\ref{T:main} makes it possible to prove (a stronger form of)
	Voisin's conjecture for
	Kummer surfaces for all values of $N$.
	
	\begin{thm}\label{T:voisin}
		Conjecture~\ref{C:voisin} is true for Kummer surfaces for all integers $N\geq
		2$.
	\end{thm}
	\begin{proof}
		Let $A$ be a  polarized abelian surface, and let $S$ be the Kummer surface
		attached to $A$. We view $S$ as the quotient of the blow-up $\widetilde{A}$ of
		$A$ along its $2$-torsion points  by the involution induced by the
		multiplication-by-$(-1)$ map on $A$. In particular, since the cohomology of
		$\widetilde{A}$ differs from that of $A$ only by Hodge classes, 
		we have the analogue of Theorem~\ref{T:gen2}  for the very general polarized
		abelian surface as long as we allow the sum to run through all cycles $Q \in
		\CH^n(A^m)$ which are products of cycles of the form $(p_i)^*L, (p_i)^*E_r,
		(p_{i,j})^*P$, where $E_r$  denote the exceptional curves of $\widetilde A$.
		As a
		consequence, one can show that the conclusion of Theorem~\ref{T:main}
		holds for $\widetilde{A}$ by working with the universal
		polarized abelian surface of degree $d^2$ with level-4 structure. (We avoid
		working with level-2
		structure in order to avoid having to deal with stacks.) In fact, quotienting
		by
		the action of multiplication-by-$(-1)$ fiber-wise, the conclusion of Theorem
		\ref{T:main} holds for the induced universal family of Kummer surfaces.
		
		Since $pr$,  $p_{\wedge^{N}}$ and $p_{\wedge^{N-1}}$ are generically
		defined, arguing as in the proof of Proposition~\ref{P:niveau}, we may
		construct
		a generically defined idempotent correspondence
		$q \in \CH^{2N}(S^N\times S^N)$ such that  
		$$q_{*} \HH^*(S^N,\Q) =  \ker \big(pr_* : (p_{\wedge^{N}})_*\HH^*(S^N,\Q)
		\to (p_{\wedge^{N-1}})_*\HH^*(S^{N-1},\Q)\big).$$
		More precisely, there is a generically defined correspondence $\gamma$ on
		$S^{N-1}\times S^N$ such that $q = \mathrm{id} - \gamma\circ p_{\wedge^{N-1}}
		\circ pr_* \circ p_{\wedge^{N}}$.  In particular, we see that 
		$$q_{*}\CH_l(S^N) \supseteq \ker\big(pr_* : (p_{\wedge^{N}})_*\CH_l(S^N) \to
		(p_{\wedge^{N-1}})_*\CH_l(S^{N-1})\big).$$
		On the other hand, defining $\pi^2_S$ to be the generically defined
		idempotent $\Delta_S - [0]\times S - S\times [0] \in \CH^2(S\times S)$, we
		have
		$$\CH_0(S)_{\mathrm{num}} = (\pi^2_S)_*\CH_0(S).$$
		Therefore, in order to prove the theorem, it is enough to establish that 
		$$\left( p_{\wedge^{N+1}} \circ (q \otimes \pi^2_S)\right)_*
		\left(\CH_l(S^{N}) \otimes \CH_0(S)\right)  = 0 \quad \text{for all } l<N.$$ 
		A cohomological calculation (as performed by Voisin \cite[p. 274]{voisin0})
		shows that 
		$$\left( p_{\wedge^{N+1}} \circ (q \otimes \pi^2_S)\right)^*
		\HH^{i,j}(S^{N+1},\Q) = 0 \quad \text{for all } i< N.$$
		Therefore, by Theorem~\ref{T:main} applied to polarized abelian
		surfaces of degree $d^2$ with level-4 structure we obtain the stronger result
		that 
		$$\left( p_{\wedge^{N+1}} \circ (q \otimes \pi^2_S)\right)_* \CH_l(S^{N+1}) 
		= 0 \quad \text{for all } l<N.$$ This concludes the proof of the theorem.
	\end{proof}

	\subsection{Varieties motivated by an abelian surface}\label{sec:motab}
	Here, we say that a smooth projective variety is  \emph{motivated by an abelian
		variety}
	$A$ if its Chow motive is isomorphic to an object in the full, thick and rigid
	subcategory of Chow motives generated by $A$. In other words, $X$ is motivated
	by $A$ if $\mathfrak{h}(X)$ is isomorphic to a direct summand of a motive of
	the
	form $\bigoplus_i \mathfrak{h}(A^{m_i})(n_i)$ for some integers $m_i\geq 0$ and
	$n_i \in \mathbb{Z}$.
	In particular, by Corollary~\ref{cor:GHCsurface}, a strong
	form
	of the generalized Hodge conjecture  holds for
	the
	powers of $X$\,; \emph{i.e.}, 	
	\begin{equation}\label{E:K}
		\mathrm{N}_H^r\HH^k(X^m,\Q) = \Gamma_*\HH^{k-2r}(B,\Q),
	\end{equation} where $B $ is a disjoint union of abelian varieties and where
	$\Gamma$ is a correspondence between $B$ and
	$X^m$.	
	Examples of varieties motivated by an abelian surface include generalized
	Kummer varieties (see \cite{xu}, and also \cite[Corollary 6.3]{ftv}).
	In particular, the following theorem applies to generalized Kummer varieties.
	\begin{thm}\label{T:app}
		Let $X$ be a smooth projective variety of dimension $d$ and let  $\gamma \in
		\mathrm{\CH}^d(X\times X)$ be a correspondence. Assume that the motive of $X$
		is
		motivated by the motive of an abelian variety $A$ of dimension $\leq 2$.
		If $\gamma^*\HH^{i,j}(X)=0$ for all $j<n$, then there exists an integer
		$N\geq 1$ such that
		$(\gamma^{\circ N})_*\CH_r(X) = 0$ for all $r<n$. In particular, if in
		addition
		$\gamma$ is an idempotent,  then $\gamma_*\CH_r(X) = 0$ for all $r<n$.
	\end{thm}
	\begin{proof}By Corollary~\ref{cor:GHCsurface}, any Hodge sub-structure of
		$\HH^*(X,\Q)$ is a $L(A)$-sub-representation. One can then proceed as in the
		proof of Theorem~\ref{T:mainLef} by invoking Proposition \ref{prop:projlef} to show that the cohomology class of $\gamma$ is a
		linear combination of cycles in $\CH^*(X\times X)$ that factor as
		$$\mathfrak{h}(X)
		\stackrel{\rho}{\longrightarrow} \mathfrak{h}(B)(n) 
		\stackrel{\zeta}{\longrightarrow} \mathfrak{h}(X),$$ 
		for some abelian
		varieties $B$ and some correspondences $\rho$ and
		$\zeta$ that belong to $\CH^*(X\times B)$ and $\CH^*(B\times X)$, respectively. 
One concludes by Kimura finite-dimensionality as for instance in the proof of
		Theorem~\ref{T:main2}(1).
	\end{proof}
	
	\begin{rmk}\label{rmk:kummerdist} In the case where $X$ is a generalized Kummer
		variety, one can be more precise. By \cite[\S 4.5]{fv}, one can define, for
		all
		integers $m\geq 0$, $\Q$-sub-algebras  $\DCH^*(X^m) \subseteq \CH^*(X^m)$ consisting of \emph{distinguished cycles} that
		map isomorphically to $\overline{\CH}^*(X^m)$ and that are compatible with
		pushforwards and pullbacks along projections. In particular, the composition of distinguished correspondences is distinguished. As such, in Theorem~\ref{T:app}, if one
		chooses $\gamma$ to be a correspondence in $\DCH^*(X\times X)$ such that
		$\gamma^*\HH^{i,j}(X)=0$ for all $j<n$, then Proposition~\ref{prop:projlef} shows that $\gamma$ is a
		linear combination of cycles in $\DCH^*(X\times X)$ that factor as
		$$\mathfrak{h}(X)
		\stackrel{\rho}{\longrightarrow} \mathfrak{h}(B)(n) 
		\stackrel{\zeta}{\longrightarrow} \mathfrak{h}(X),$$ 
		for some abelian
		varieties $B$ and some correspondences $\rho$ and
		$\zeta$ that belong to $\DCH^*(X\times B)$ and $\DCH^*(B\times X)$, respectively. One concludes that  $\gamma_*\CH_r(X) = 0$ for all
		$r<n$.
	\end{rmk}

	\begin{rmk}
		The results of Sections~\ref{s:lin} and \ref{s:voisin} could have been
		established by referring to
		Theorem~\ref{T:app} instead of Theorem~\ref{T:main}. We chose to refer to
		Theorem
		\ref{T:main} (which is concerned with generically defined cycles) because it
		is more elementary and does not appeal
		to
		Abdulali's theorem on the generalized Hodge conjecture for powers of CM
		abelian
		surfaces. Moreover the approach using generically defined cycles seems more
		natural and is probably better suited to adapt to other situations. Nonetheless,
		Theorem~\ref{thm:sympGK} below will use the full
		strength of Theorem~\ref{T:app}.
	\end{rmk}

	\subsection{Finite-order symplectomorphisms on generalized Kummer
		varieties}\label{sec:symplecto}
	
	Let $(X,\omega)$ be a symplectic variety, that is, a smooth projective variety
	equipped with a nowhere degenerate $2$-form $\omega$.
	A symplectomorphism of $(X,\omega)$ is an automorphism $f : X \to X$ such that
	$f^*\omega = \omega$. If $X$ is irreducible symplectic,  it is
	expected as part of the Bloch conjectures that symplectomorphisms act unipotently on the Chow group of $0$-cycles, and, due to the probable \emph{distinguishedness} of symplectomorphisms in the sense of \cite{fv}, it is in fact expected that symplectomorphisms act as the identity on the Chow group of $0$-cycles. Most notably, this was established
	for symplectic involutions on K3 surfaces by Voisin \cite{voisink3} and
	extended to finite-order symplectomorphisms on K3 surfaces by Huybrechts
	\cite{huybrechts}. This was also established for polarized symplectomorphisms
	of Fano varieties of lines on smooth cubic fourfolds by Fu \cite{fu}, that is,
	for symplectomorphisms that preserve a given polarization.
	We extend that type of results to generalized Kummer varieties.
	
	\begin{thm}\label{thm:sympGK}
		Let $A$ be an abelian surface and let  $f$ be a symplectomorphism of
		the
		generalized Kummer variety $K_n(A)$. Then $f_* : \CH_0(K_n(A)) \to
		\CH_0(K_n(A))$ is unipotent. In particular, if $f$  is a finite-order symplectomorphism, then $f_* : \CH_0(K_n(A)) \to
		\CH_0(K_n(A))$ is the identity.
	\end{thm}
	\begin{proof} Since $\HH^{2i,0}(K_n(A)) = \HH^0(\Omega_{K_n(A)}^{2i}) =
		\C\omega^i$, and
		by definition of a symplectomorphism, $f^*$ acts as the identity on
		$\HH^{2i,0}(K_n(A))$ for all $i$. Therefore, by Theorem~\ref{T:app},
		$f-\mathrm{id}$ acts nilpotently on $\CH_0(K_n(A))$. Suppose now that $f$
		has
		finite order. In particular a positive power of $f$ acts as the identity on
		$\CH_0(K_n(A))$. Since the gcd of the polynomials $X^n-1$ and $(X-1)^N$ is
		$X-1$, we find that $f - \mathrm{id}$ acts as zero on $\CH_0(K_n(A))$.
	\end{proof} 

	Finally we note that 
	if $f$ is a symplectomorphism of the generalized Kummer variety $K_n(A)$
	induced by a symplectomorphism of $A$, then Pawar \cite{pawar} showed that
	$f_*$ acts as the identity
	on $\CH_0(K_n(A))_{(2n)}$ (as defined in \S \ref{s:lin}). 
We can extend Pawar's result and show that
	$f_*$ acts as the identity
	on the whole of $\CH_0(K_n(A))$\,:
	
	\begin{prop}\label{prop:genkuminduced}
Suppose $f$ is a symplectomorphism of the generalized Kummer variety $K_n(A)$
induced by a symplectomorphism of $A$. Then
$f_*$ acts as the identity
on $\CH_0(K_n(A))$.
	\end{prop}
	\begin{proof}
One uses Remark~\ref{rmk:kummerdist} and
notes that the graph of a symplectomorphism induced by a symplectomorphism of
$A$
belongs to the sub-algebra $\DCH^*(K_n(A)\times K_n(A))$ defined in \cite[\S
5.5]{fv}.
	\end{proof}

\end{document}